\documentclass[12pt]{article}
\usepackage[T1]{fontenc}
\usepackage{amssymb, amsmath, amsthm}
\usepackage[all]{xy}
\usepackage{newtxtext,stmaryrd}
\usepackage{hyperref}
\usepackage{accents}
\usepackage{tikz-cd}
\usepackage{hhline}
\usepackage{multicol}
\usepackage[scr=esstix]{mathalpha}

\usetikzlibrary{backgrounds,matrix}

\DeclareFontEncoding{OT2}{}{}

\newcommand{\mr}[1]{\mathrm{#1}}
\newcommand{\mf}[1]{\mathfrak{#1}}
\newcommand{\mc}[1]{\mathcal{#1}}
\newcommand{\mb}[1]{\mathbb{#1}}

\newcommand{\ms}[1]{\mathsf{#1}}

\newcommand{\Z}{\mb{Z}}
\newcommand{\Q}{\mb{Q}}
\newcommand{\zp}{\mb{Z}_p}
\newcommand{\qp}{\mb{Q}_p}

\newcommand{\F}{\mb{F}}
\newcommand{\R}{\mb{R}}
\newcommand{\C}{\mb{C}}

\newcommand{\smatrix}[1]{\left(\begin{smallmatrix}#1\end{smallmatrix}\right)}

\DeclareMathOperator{\Hom}{Hom} \DeclareMathOperator{\Aut}{Aut}
 \DeclareMathOperator{\Gal}{Gal}
 \DeclareMathOperator{\res}{res}

\DeclareMathOperator{\Spec}{Spec}

\DeclareMathOperator{\diag}{diag}

\DeclareMathOperator{\cor}{cor}

\newcommand{\Cl}{\mr{Cl}}

\newcommand{\et}{\mr{\acute{e}t}}

\newcommand{\BS}{\mr{BS}}

\newcommand{\GL}{\mathrm{GL}}
\newcommand{\ab}{\mathrm{ab}}

\newcommand{\gm}{\mb{G}_m}

\newtheorem{theorem}{Theorem}[subsection]
\newtheorem{proposition}[theorem]{Proposition}
\newtheorem{lemma}[theorem]{Lemma}
\newtheorem{corollary}[theorem]{Corollary}
\newtheorem*{thm}{Theorem}

\theoremstyle{definition}
\newtheorem{definition}[theorem]{Definition}

\theoremstyle{remark}
\newtheorem{remark}[theorem]{Remark}
\newtheorem*{ack}{Acknowledgments}

\newcounter{countii}

\numberwithin{equation}{section}

\newcommand\extrafootertext[1]{%
    \bgroup
    \renewcommand\thefootnote{\fnsymbol{footnote}}%
    \renewcommand\thempfootnote{\fnsymbol{mpfootnote}}%
    \footnotetext[0]{#1}%
    \egroup
}

\begin{document}

\title{Eisenstein cocycles for imaginary quadratic fields}
\author{Emmanuel Lecouturier, Romyar Sharifi, Sheng-Chi Shih, Jun Wang}
\date{}
\maketitle

\begin{abstract}
We construct Eisenstein cocycles for arithmetic subgroups of $\GL_2$ of imaginary quadratic fields valued in second $K$-groups of products of two CM elliptic curves.  We use these to construct maps from the first homology groups of Bianchi spaces to corresponding second $K$-groups of ray class fields and to verify the Eisenstein property of these maps for prime-to-level Hecke operators.
\end{abstract}

\section{Introduction}

\subsection{Background and the main theorem}

The papers \cite{busuioc} and \cite{sharifi} defined explicit maps from first homology groups of modular curves to second $K$-groups of cyclotomic integer rings. For a positive integer $N$, the map has the form
$$
	\Pi_N^{\circ} \colon H_1(X_1(N),C,\Z[\tfrac{1}{2}])^+ \to (K_2(\Z[\mu_N,\tfrac{1}{N}]) \otimes \Z[\tfrac{1}{2}])^+,
$$
where $C$ is the set of cusps not lying over  $\infty \in X_0(N)$, and
a superscript ``$+$'' denotes fixed part under complex conjugation (or a projection to said part). 
It carries explicit generators of homology $[u:v]$ for $u, v \in \Z/N\Z - \{0\}$ with $(u,v) = 1$, known as Manin symbols \cite{manin},
to Steinberg symbols of cyclotomic $N$-units via the explicit recipe
$$
	\Pi_N^{\circ}([u:v]^+) = \{1-\zeta_N^u,1-\zeta_N^v\}^+,
$$
where $\zeta_N$ is a fixed primitive $N$th root of unity.
A map from homology to a formally defined and infinite but related group had been considered by Goncharov in \cite[Section 4]{gon-double}.

The second author conjectured that $\Pi_N^{\circ}$ should be Eisenstein in the sense that 
\begin{equation} \label{Eis_away_level}
	\Pi_N^{\circ} \circ (T_{\ell} - \ell - \langle \ell \rangle) = 0
\end{equation}
for all primes $\ell \nmid N$
and 
$$
	\Pi_N^{\circ} \circ (U^*_{\ell} - 1) = 0
$$
for all primes $\ell \mid N$, where $U^*_{\ell}$ denotes the dual Hecke operator. 
Fukaya and Kato proved this for the projection of 
$\Pi_N^{\circ}$ to the $p$-part of the second $K$-group for each $p \mid N$
by realizing $\Pi_N^{\circ}$ as the specialization at infinity of a Hecke-equivariant zeta map \cite[Theorem 5.3.3]{fk}. 

The restriction of $\Pi_N^{\circ}$ provides a map 
$$
	\Pi_N \colon H_1(X_1(N),\Z[\tfrac{1}{2}])^+ \to (K_2(\Z[\mu_N]) \otimes \Z[\tfrac{1}{2}])^+,
$$
which is then conjecturally also Eisenstein. Moreover, the second author has conjectured that the map induced
by $\Pi_N$ on its quotient by a corresponding Eisenstein ideal is actually an isomorphism (see \cite{sharifi,sv}).
On the other hand, $\Pi_N^{\circ}$ is not an isomorphism for certain $N$.
 
In \cite{sv}, the second author and Venkatesh proved that $\Pi_N$ is Eisenstein away from the level $N$, i.e., satisfies \eqref{Eis_away_level}
for $\ell \nmid N$. To do so, they realized $\Pi_N$ as a specialization at an $N$-torsion point
of the restriction to $\Gamma_1(N)$ of a cocycle 
$$
	\Theta_N \colon \GL_2(\Z) \to K_2(\Q(\gm^2))/\langle \{ -z_1,-z_2 \} \rangle,
$$
where $z_i$ is the $i$th coordinate function on $\gm^2$. 
The work of the first and fourth authors \cite{lw} combines the results of \cite{fk} and \cite{sv} together with
level compatibilities to obtain the full Eisenstein property of $\Pi_N$ upon inverting $3$.

In \cite[Section 4.2]{fks-survey}, Fukaya, Kato, and the second author directly raised the question that had been floating out there
as to whether an analogous Eisenstein map exists between the homology of a Bianchi space and the second $K$-group of a ray class field of an imaginary quadratic field $F$, with elliptic units replacing cyclotomic units. Here is a representative theorem from Section \ref{cohom_Bianchi}.

\begin{thm}
	Let $F$ be an imaginary quadratic field with integer ring $\mc{O}$, and let $\mc{N}$ be an ideal of $\mc{O}$
	such that the canonical map $\mc{O}^{\times} \to (\mc{O}/\mc{N})^{\times}$ is injective. Let $Y_1(\mc{N})$
	be the Bianchi space for $F$ with $\Gamma_1(\mc{N})$-level structure. Let $\mf{c}$ be 
	a proper ideal of $\mc{O}$ prime to $\mc{N}$. Let $\mc{O}'(\mc{N})$ denote the ring of integers $\mc{O}(\mc{N})$
	of the ray class field of $F$ of modulus $\mc{N}$ if $\mc{N}$ is not a prime
    power and $\mc{O}(\mc{N})[\tfrac{1}{\mc{N}}]$ otherwise.
	Then there exists a homomorphism
	$$
		{}_{\mf{c}} \Pi_{\mc{N}} \colon H_1(Y_1(\mc{N}),\Z[\tfrac{1}{30}]) \to K_2(\mc{O}'(\mc{N})) \otimes \Z[\tfrac{1}{30}]
	$$
	that is Eisenstein away from $\mc{N}$, which is a specialization at $\mc{N}$-torsion of an Eisenstein tuple of cocycles
	valued in second $K$-groups of products of two elliptic curves with CM by $\mc{O}$. 
	Given a second proper ideal $\mf{d}$ prime to $\mc{N}$, we have
	$$
		(N\mf{d}^2 - \mc{R}(\mf{d})) \circ {}_{\mf{c}} \Pi_{\mc{N}}  
		= (N\mf{c}^2 - \mc{R}(\mf{c})) \circ {}_{\mf{d}} \Pi_{\mc{N}} 
	$$
	for $N$ the absolute norm and $\mc{R}$ the Artin reciprocity
    map on the ray class group of modulus $\mc{N}$.
\end{thm}

Of course, were it not for the imprecise caveat about its definition as a specialization, the zero map would automatically satisfy the theorem, but we do not define ${}_{\mf{c}} \Pi_{\mc{N}}$ as such! Though we have not attempted it, to see that ${}_{\mf{c}} \Pi_{\mc{N}}$ is nonzero should be possible in certain cases by comparison with the maps $\Pi_N$ for modular curves, which can often be shown to be nonzero.
In Theorem \ref{Eis_map_homol}, we show under certain conditions that ${}_{\mf{c}} \Pi_{\mc{N}}$ factors through the homology of the Satake compactification $X_1(\mc{N})$ given by adjoining the cusps and that ${}_{\mf{c}} \Pi_{\mc{N}}$ arises from a map $\Pi_{\mc{N}}$ independent of $\mf{c}$, but without connecting it to any expected explicit formulas. 

The idea that a map $\Pi_{\mc{N}}$ should exist, sans its Eisenstein
property, was hinted at by Goncharov's construction of a related map to a $\Q$-vector space in \cite[Theorem 5.7]{gon-Euler} for prime level and
$F \in \{\Q(i),\Q(\mu_3)\}$. The work of Calegari and Venkatesh \cite[Theorem 4.5.1(ii)]{cv} addressed the Eisenstein property: for $F$ such that $H_1(Y_0(1),\Z)$ is torsion and prime $\mf{q}$ of $\mc{O}$, they constructed a surjective map from a subgroup of $H_1(Y_0(\mf{q}),\Z)$ to $K_2(\mc{O}) \otimes \Z[\frac{1}{6}]$ and showed it factors
through the Eisenstein quotient. Computations of homology groups of Bianchi spaces modulo Eisenstein ideals were performed in the Ph.D. thesis of Powell \cite{powell} as a test case for the speculations of \cite{fks-survey}. 

There have recently been a variety of constructions of what are commonly referred to as Eisenstein cocycles. Among them are the Sczech-style cocycles of Bergeron, Charollois, and Garcia \cite{bcg} and Fl\'orez, Karabulut, and Wong \cite{fkw} 
for $\GL_n$ over an imaginary quadratic field and the equivariant coherent  classes of Kings and Sprang \cite{ks} over CM fields. These fascinating works aim at formulas for critical values of $L$-functions of Hecke characters related to Eisenstein-Kronecker series and Dedekind sums. There exist tentative connections between these works and ours, but the settings, methods, and aims are markedly different.

\subsection{Outline of the construction}

To construct our maps, we refine the approach of \cite{sv}. Let us outline this, highlighting a variety of issues we must overcome.
Just as cyclotomic units are the specializations of the function $1-z$ at $z$ at roots of unity, elliptic units are
specializations of theta functions on a CM elliptic curve at a torsion point. With this in mind, and similarly to the case of a square of an elliptic curve considered in \cite{sv}, we replace $\gm^2$ by a product $E \times E'$ over $L$ of two CM elliptic curves over an abelian extension $L$ of an imaginary quadratic field $F$. 

Much as the case of $\gm^2$, our cocycle is constructed using a two-term Gersten-Kato complex $\ms{K} = [\ms{K}_2 \to \ms{K}_1 \to \ms{K}_0]$
of the form
$$ 
	K_2(L(E \times E')) \to \bigoplus_D L(D)^{\times} \to \bigoplus_x \Z
$$
with the sums taken over irreducible divisors $D$ and zero-cycles $x$ on $E \times E'$, respectively.
Specifically, we use the trace-fixed subcomplex $\ms{K}^{(0)}$ of $\ms{K}$, after inverting $5!$. This fits in an exact sequence
$$
	0 \to \ms{K}_2^{(0)} \to \ms{K}_1^{(0)} \to \ms{K}_0^{(0)} \to \Z[\tfrac{1}{30}] \to 0,
$$
the last map being the degree map on zero-cycles. Let $\mf{c}$ be a proper, nonzero ideal of the ring of integers $\mc{O}$ of $F$.
Analogously to Kato's construction of theta functions as norm invariant elements with a given degree zero divisor \cite{kato} (or $1-z$ as as the trace-fixed unit on $\mb{A}^1-\{1\}$ with divisor $1$), we start with a degree $0$ formal sum $e_{\mf{c}} \in \ms{K}_0^{(0)}$ of $\mf{c}$-torsion points on $E \times E'$ that is fixed by the action of the subgroup $\Gamma$ of $\GL_2(F)$ of automorphisms of $E \times E'$. Pulling back by the resulting map $\Z \to \ms{K}_0^{(0)}$ gives an extension of $\Z$ by $\ms{K}_2^{(0)}$ that is the class of a $1$-cocycle
$$
	{}_{\mf{c}} \Theta_{E \times E'} \colon \Gamma \to K_2(L(E \times E')) \otimes \Z[\tfrac{1}{30}].
$$
The particular choice of cocycle is given by a certain sum of theta functions 
on $E$ and $E'$ pulled back to divisors on $E \times E'$ with residue $e_{\mf{c}}$.

Now, unless the class number of $F$ is $1$, there
is no action of arbitrary Hecke operators on the first $\Gamma$-cohomology of $K_2(L(E \times E'))$. Rather, we define an action on the collection of such cohomology groups as $E$ and $E'$ run over the isomorphism classes of elliptic curves with CM by $\mc{O}$, noting here that $\Gamma$ varies with $E \times E'$. That is, in Section \ref{Hecke} we define 
the notion of a $\Delta$-module system and Hecke operators on the cohomology group of such a system, which we employ for this purpose. 
In Proposition \ref{big_cocycle_Eisenstein}, we show that the tuple ${}_{\mf{c}} \Theta$ of classes is Eisenstein with respect to Hecke operators attached to prime ideals. 

Though representative elliptic curves can be defined over the Hilbert class field of $F$, the torsion on such a curve does not always define an abelian extension of $F$. Fortunately, as in \cite{deshalit}, we may choose representative elliptic curves with this property over any ray class field $L$ of level $\mf{f}$ such that $\mc{O}^{\times}$ injects into $(\mc{O}/\mf{f})^{\times}$. We restrict a slight modification of ${}_{\mf{c}} \Theta_{E \times E'}$
to a congruence subgroup of $\Gamma$ that one might typically denote $\Gamma_0(\mc{N})$, for an $\mc{O}$-ideal $\mc{N}$ prime to $\mf{c}$.
This takes values in a direct limit of motivic cohomology groups $H^2(U,2)$, where $U$ is an open in $E \times E'$ containing all points $(0,Q)$ with $Q$ a primitive $\mc{N}$-torsion point on $E'$. We can then pull back using a good choice of $Q$ to obtain a specialized cocycle
$$
	{}_{\mf{c}} \Theta_{\mc{N}, E \times E'} \colon \Gamma_0(\mc{N}) \to K_2(F(\mc{N} \cap \mf{f})) \otimes \Z[\tfrac{1}{30}].
$$
In Proposition \ref{indep_of_f}, we show that by varying the $\mf{f}$ implicit in the definition, we obtain a cocycle valued in $K_2(F(\mc{N})) \otimes \Z[\tfrac{1}{30}]$. The tuple ${}_{\mf{c}} \Theta_{\mc{N}}$ of classes as we vary $E \times E'$ remains Eisenstein away from the level. In Proposition \ref{integral_image}, we show that this cocycle takes values in $K_2(\mc{O}'(\mc{N})) \otimes \Z[\tfrac{1}{30}]$ using motivic cohomology over Dedekind schemes.

We also treat the dependence of our tuples of classes on the auxiliary ideal $\mf{c}$ in Corollary \ref{compare_aux_ideal}, which requires a 
subtle analysis of pullbacks of our cocycles. As with elliptic units that are the pullbacks of theta functions at torsion points, one might hope that there exists a tuple $\Theta_{\mc{N}}$ with $(N\mf{c}^2 - \mc{R}(\mf{c})) \circ \Theta_{\mc{N}} = {}_{\mf{c}} \Theta_{\mc{N}}$ up to some controllable denominators. Here, of course, the consideration of denominators is complicated by the fact that we are working with the finite group $K_2(\mc{O}'(N))$. In Theorem \ref{unmodified_eigensp}, we show that
$\Theta_{\mc{N}}$ exists upon tensor product with $\zp$ and the taking of an eigenspace for a character $\chi$ of the prime-to-$p$ part of $(\mc{O}/\mc{N})^{\times}$, avoiding one particular choice of $\chi$.

Next, we note that the restrictions of our cocycles to $\Gamma_1(\mc{N})$-type subgroups are homomorphisms, as the action of $\Gamma_0(\mc{N})$ on $K_2(F(\mc{N}))$ is through the Artin reciprocity map applied to the lower right-hand corner of such a matrix. We note that there are 
$h = |\Cl(F)|$ isomorphism classes of elliptic curves with CM by $\mc{O}$, and our tuple ${}_{\mf{c}} \Theta$ is of $h^2$ cocycles corresponding
to the pairs of representatives of these classes. On the other hand, the Bianchi space $Y_1(\mc{N})$ of level $\mc{N}$ has $h$ connected components. To define a map on $H_1(Y_1(\mc{N}),\Z[\frac{1}{30}])$, we take $E$ to the the representative curve
satisfying $E(\C) \cong \C/\mc{O}$ and vary $E'$. We then obtain the Eisenstein map ${}_{\mf{c}} \Pi_{\mc{N}}$ in the theorem (see Section \ref{cohom_Bianchi}). In doing so, we account for the fact that the Hecke operators we defined, which are well-suited for describing the action on our tuple of $h^2$ cocycles, do not agree with the usual Hecke operators on cohomology.

Unlike the cocycles of \cite{sv}, we do not expect the cocycles defining ${}_{\mf{c}} \Theta$ to be parabolic, which is to say that they are not coboundaries on parabolic subgroups. However, we do expect parabolicity to hold for the specialization ${}_{\mf{c}} \Theta_{\mc{N}}$. In Proposition \ref{Eis_part_BS}, we prove a parabolicity result for the resulting maps on homology by comparison of the eigenvalues of our Hecke operators on ${}_{\mf{c}} \Pi_{\mc{N}}$ with the eigenvalues of Hecke operators on the homology of the boundary of the Borel-Serre compactification of $Y_1(\mc{N})$. In this setting, parabolicity means that the maps ${}_{\mf{c}} \Pi_{\mc{N}}$ on the homology of $Y_1(\mc{N})$ factor through the homology of the Satake compactification $X_1(\mc{N})$. Putting this together with our independence from $\mf{c}$ result, we obtain in Theorem \ref{Eis_map_homol} a map $\Pi_{\mc{N}}^{\chi}$ on the $\chi$-eigenspace of $H_1(X_1(\mc{N}),\zp)$, where $\mc{N}$ is divisible by at most one power of each prime over $p$ and $\chi$ is a character on $(\mc{O}/\mc{N})^{\times}$ not on a certain short list.

Ideally, we would have an Eisenstein map $\Pi_{\mc{N}}^{\circ}$ defined on the homology of $X_1(\mc{N})$ relative
to certain cusps, taking Manin-type symbols (cf.~Cremona's work \cite{cremona}) to Steinberg
symbols of elliptic units. There are several obstacles, not least that our proof that $\Pi_{\mc{N}}$ exists in some cases as
a map on the homology of the compactification $X_1(\mc{N})$ is indirect and does not follow from a statement about the tuple ${}_{\mf{c}} \Theta$. 
On the other hand, there are no evident Steinberg relations among general elliptic $\mc{N}$-units, so no obvious way in which to define such a map on the larger relative homology group directly. Moreover, Cremona's symbols are defined only in the Euclidean setting, and most generalizations don't seem ideally suited to such a treatment. Additionally, the elliptic units we would
want to consider are not in general true elements of a unit group, but rather roots thereof, and thus need in general to be modified by some auxiliary ideal $\mf{c}$ that complicates the derivation of formulas in our approach even in the Euclidean setting. Thus, we have omitted any treatment of such formulas in this paper, despite the fact that the connection with Steinberg symbols of elliptic units is almost implicit in our constructions. We have some novel ideas for overcoming many of the obstacles we have just described, but this is left for future work.

\begin{ack}
    The first author completed most of this project at the Yau Mathematical Sciences Center of Tsinghua University and was also funded by the Beijing Institute of Mathematical Sciences and Applications. The final details of this project were completed at the Institute for Theoretical Sciences of Westlake University. He was also supported by the Institute for Advanced Study in Princeton during Spring 2022 and funded by the National Natural Science Foundation of China under Grant No.~12050410242. 
    
    The second author thanks Akshay Venkatesh for his insights during their joint work and Takako Fukaya and Kazuya Kato for early conversations on this subject. He also thanks everyone who encouraged him to explore these maps over many years, including Adebisi Agboola and Cristian Popescu. This material was based in part upon work of the second author supported by the National Science Foundation (NSF) under Grant No.~DMS-2101889. Part of this research was performed while the second author was visiting the Simons Laufer Mathematical Sciences Institute, which is supported by NSF Grant No.~DMS-1928930.

    The third author would like to thank Adel Betina and Ming-Lun Hsieh for helpful discussions.
    
    The fourth author would like to express gratitude to Taiwang Deng, Yangyu Fan, and Yichao Zhang for helpful discussions and insights regarding Bianchi modular forms and Eisenstein cohomology. The fourth author is supported by the National Natural Science Foundation of China under Grant No.~12331004.
\end{ack}

\section{Hecke actions on cohomology} \label{Hecke}

\subsection{Module systems} \label{modsys}

Let $F$ be a number field and $\mc{O}$ be its ring of integers. For any nonzero ideal $\mc{N}$ of $\mc{O}$, we denote by $\Cl_{\mc{N}}(F)$ the ray class group of modulus $\mc{N}$ of $F$. 
Let $h$ be the order of the class group $\Cl(F)$ of $F$. Set $I = \{1,\ldots,h \}$, and
fix representative $\mc{O}$-ideals $\mf{a}_r$ for $r \in I$ of $\Cl(F)$ with $\mf{a}_1 = \mc{O}$. 
For each pair $(r,s) \in I^2$, set $\mf{a}_{r,s} = \mf{a}_s\mf{a}_r^{-1}$ for brevity. 
Fix $n \ge 1$.

Let $R$ denote the profinite completion of $\mc{O}$.
For each $r \in I$, fix a finite id\`{e}le $\alpha_r \in R$ representing $\mf{a}_r$. 
Set $\mc{G} = \GL_n(\mb{A}_F^f)$
and fix an open subgroup $U$ of the profinite group $\GL_n(R)$.
Let $\tilde{\Delta}$ be a submonoid of $\mc{G} \cap \mr{M}_n(R)$ containing $U$. For $i = (i_1, \ldots, i_n) \in I^n$, set 
$$
	x_i = \diag(\alpha_{i_1}, \ldots, \alpha_{i_n}).
$$
Given also $j \in I^n$, set $\tilde{\Delta}_{i,j} = x_i^{-1}\tilde{\Delta}x_j$.
We then set $\tilde{\Delta}_i = \tilde{\Delta}_{i,i}$ and $U_i = x_i^{-1}Ux_i \subset \tilde{\Delta}_i$.

Now set $G = \GL_n(F)$. 
For any $i = (i_1, \ldots, i_n)$ and $j = (j_1, \ldots, j_n) \in I^n$, let 
$\Delta_{i,j} = \tilde{\Delta}_{i,j} \cap G$. In the case that $\tilde{\Delta} = \mc{G} \cap \mr{M}_n(R)$, we have
$$
    \Delta_{i,j} =
	\{ (a_{u,v})_{u,v} \in G \mid a_{u,v} \in \mf{a}_{i_u,j_v} \text{ for all } 1 \le u,v \le n \},
$$
independent of our choice of id\`{e}les.
Note that for $k \in I^n$, we have $\Delta_{i,j}\Delta_{j,k} \subseteq \Delta_{i,k}$.
Set $\Delta_i = \Delta_{i,i}$ and $\Gamma_i = U_i \cap G$. 
The group $\Gamma_i$ is commensurable with $\GL_n(\mc{O})$, so its commensurator in $G$ is $G$ by \cite[Lemma 3.10]{shimura}. For $i = (1,\ldots,1)$, we have that $\Delta_i \subseteq \mr{M}_n(\mc{O}) \cap \GL_n(F)$ and $\Gamma_i = \GL_n(\mc{O}) \cap U$.

\begin{definition}\
\begin{enumerate}
	\item[a.] A \emph{$\Delta$-module system} $A$ indexed by $J \subseteq I^n$ is a collection of abelian groups $A_i$ for $i \in J$ such that each
	$g \in \Delta_{i,j}$ with $i,j \in J$ provides a homomorphism
	$g \colon A_j \to A_i$, which satisfy
    \begin{enumerate}
        \item[i.] if $g \in \Delta_{i,j}$ and $g' \in \Delta_{j,k}$ for 
        $i, j, k \in J$, then $g \circ g' = gg' \colon A_k \to A_i$, and
        \item[ii.] the identity matrix $1_n$ provides the identity 
        homomorphism on $A_i$ for each $i \in J$.
    \end{enumerate}
	\item[b.]A \emph{homomorphism $\phi \colon A \to B$ of $\Delta$-module systems} indexed by $J$ is a collection of homomorphisms
	$\phi_i \colon A_i \to B_i$ for $i \in J$ such that  $\phi_i \circ g = g \circ \phi_j \colon A_j \to B_i$ for all $g \in \Delta_{i,j}$.
\end{enumerate}
\end{definition}

In \cite{rw}, Rhie and Whaples defined a right action of an abstract Hecke ring on group cohomology with coefficients in a module. Fixing a $\Delta$-module system $A$ indexed by some $J$, this applies in particular to give a right action of the abstract Hecke algebra for the Hecke pair $(\Delta_i,\Gamma_i)$ on $H^q(\Gamma_i,A_i)$ for all $q \ge 0$ and $i \in J$. 
The Hecke operator that we define of an element of $\Delta_{i,j}$ for $i,j \in J$ agrees with Rhie and Whaples's operator for the inverse matrix in the setting $i = j$ where our constructions can be compared (i.e., they work with $G$-modules, not $\Delta$-module systems). Their action arises from an action on homogeneous cochains, while we describe the corresponding action on inhomogeneous cochains.

Let us work somewhat more generally. We fix a second open subgroup $U'$ of $\GL_n(R)$ contained in $\tilde{\Delta}$ and then set $U'_i = x_i^{-1}U'x_i$ and $\Gamma'_i = U'_i \cap G$ for $i \in J$.
Given any $g \in \Delta_{i,j}$ for $i,j \in J$, we may decompose
the double coset $\Gamma'_i g \Gamma_j$ as a finite union
\begin{equation} \label{cosetreps}
	\Gamma'_i g \Gamma_j = \coprod_{t=1}^v g_t \Gamma_j
\end{equation}
for some $g_t \in \Delta_{i,j}$ for $1 \le t \le v$ and some $v \ge 1$.

\begin{definition}
	Given a choice of double coset decomposition for $g \in \Delta_{i,j}$ as in \eqref{cosetreps},
	we define the \emph{Hecke operator} of $\Gamma'_i g \Gamma_j$ on $f \colon \Gamma_j^q \to A_j$ as $T(g)f \colon (\Gamma'_i)^q \to A_i$ given by
	$$
		(T(g)f)(\gamma) = \sum_{t=1}^v g_{\sigma(t)} f(\mu_t),
	$$
	where for $\gamma = (\gamma_1, \ldots, \gamma_q) \in (\Gamma'_i)^q$, the elements 
	$\sigma \in S_v$ and $\mu_t \in \Gamma_j^q$ for $1 \le t \le v$ are defined as follows: 
	recursively setting $h_t^{(q)} = g_t$ and 
	$$
		\gamma_w h_t^{(w)} = h_t^{(w-1)} \mu_{t,w}
	$$ 
	with $\mu_{t,w} \in \Gamma_j$ and $h^{(w)}_t \in \{g_1, \ldots, g_v\}$ for $1 \le w \le q$, we take 
	$\mu_t = (\mu_{t,1}, \ldots, \mu_{t,q})$ and let $\sigma \in S_v$ be the unique permutation such that $g_{\sigma(t)} = h_t^{(0)}$ for
	each $1 \le t \le v$. 
\end{definition}

In general, the cochain $T(g)f$ depends upon the set of representatives $\{g_1, \ldots, g_v\}$ of $\Gamma'_i g \Gamma_j$.
However, fixing such a choice, one sees easily that $T(g)$ defines a map of cochain complexes, which also follows from the proof of the next result. This tells us that $T(g)$ is compatible with the connecting maps arising from short exact sequences of $\Delta$-module systems.

\begin{proposition} \label{Hecke_action}
	The above operation on cochains induces a homomorphism 
	$$
		T(g) \colon H^q(\Gamma_j,A_j) \to H^q(\Gamma'_i,A_i)
	$$ 
	that is independent of the choice of representatives in \eqref{cosetreps}.
\end{proposition}

\begin{proof}
	For all $\gamma \in \Gamma'_i$ and $1 \le t \le v$, write 
	$$
		\gamma g_t = g_{\sigma_{\gamma}(t)} \tau_t(\gamma)
	$$
	for some $\tau_t(\gamma) \in \Gamma_j$ and permutation $\sigma_{\gamma} \in S_v$. One can easily check that
	$\sigma_{\gamma \gamma'} = \sigma_{\gamma} \sigma_{\gamma'}$ and 
	$\tau_t(\gamma \gamma') = \tau_{\sigma_{\gamma'}(t)}(\gamma) \tau_t(\gamma')$ for $\gamma, \gamma' \in \Gamma'_i$. 

	For $F \in \Hom_{\Z[\Gamma_j]}(\Z[\Gamma_j^{q+1}],A_j)$, we define
	$T(g)F \in \Hom_{\Z[\Gamma'_i]}(\Z[(\Gamma'_i)^{q+1}],A_i)$ by
	\begin{equation} \label{Hecke_homog}
		T(g)F(\gamma_1, \ldots, \gamma_{q+1}) = \sum_{t=1}^v g_t F(\tau_{\sigma_{\gamma_1}^{-1}(t)}(\gamma_1), 
		\ldots, \tau_{\sigma_{\gamma_{q+1}}^{-1}(t)}(\gamma_{q+1}))
	\end{equation}
	for $g \in \Delta_{i,j}$ and $\gamma_1, \ldots, \gamma_{q+1} \in \Gamma'_i$. 
	
	We compare $T(g)$ with an operator
	$$
		S(g) \colon \Hom_{\Z[\Gamma_j]}(\Z[G^{q+1}],A_j) \to \Hom_{\Z[\Gamma'_i]}(\Z[G^{q+1}],A_i)
	$$ 
	given on $F' \in \Hom_{\Z[\Gamma_j]}(\Z[G^{q+1}],A_j)$ by
	$$
		S(g)F'(\delta_1, \ldots, \delta_{q+1}) = \sum_{t=1}^v g_t F'(g_t^{-1}\delta_1, \ldots, g_t^{-1}\delta_{q+1})
	$$
	for $\delta_1, \ldots, \delta_{q+1} \in G$, which clearly commutes with the standard differentials and 
	is independent of all choices.
	
	Define 
	$$
		\Pi^q \colon \Hom_{\Z[\Gamma_j]}(\Z[\Gamma_j^{q+1}],A_j) \to \Hom_{\Z[\Gamma_j]}(\Z[G^{q+1}],A_j)
	$$ 
	on $F \in \Hom_{\Z[\Gamma_j]}(\Z[\Gamma_j^{q+1}],A_j)$ by $\Pi^q(F) = F \circ \pi^{q+1}$, where
	$\pi \colon G \to \Gamma_j$ is given by 
	$$
		\pi(h) = h \cdot s(\Gamma_j h)^{-1}
	$$
	for a chosen section $s$ of the canonical surjection $G \to \Gamma_j \backslash G$ that contains
	$g_t^{-1}$ for $1 \le t \le v$ in its image. These (noncanonical) maps $\Z[G^{q+1}] \to \Z[\Gamma_j^{q+1}]$ give a map of augmented
	$\Z[\Gamma_j]$-projective resolutions of $\Z$ for the standard differentials, so $\Pi^{\cdot}$ is a quasi-isomorphism.
    
    Let $F\in \Hom_{\Z[\Gamma_j]}(\Z[\Gamma_j^{q+1}],A_j)$, and let $F' = \Pi^q(F)$.
	We have
	\begin{align*}
		S(g)F'(\gamma_1, \ldots, \gamma_{q+1}) 
		&= \sum_{t=1}^v g_t F'(\tau_{\sigma_{\gamma_1}^{-1}(t)}(\gamma_1)g_{\sigma_{\gamma_1}^{-1}(t)}^{-1},
		\ldots,\tau_{\sigma_{\gamma_{q+1}}^{-1}(t)}(\gamma_{q+1})g_{\sigma_{\gamma_{q+1}}^{-1}(t)}^{-1})\\
		&= \sum_{t=1}^v g_t F(\tau_{\sigma_{\gamma_1}^{-1}(t)}(\gamma_1), \ldots, \tau_{\sigma_{\gamma_{q+1}}^{-1}(t)}(\gamma_{q+1})),
	\end{align*}
	where the first equality follows by definition and the second since $\pi(\mu g_t^{-1}) = \mu$ for $\mu \in \Gamma_j$ and $1 \le t \le v$. 
	Since this restriction also induces a quasi-isomorphism of complexes in the opposite direction to $\Pi^{\cdot}$,
	it follows that $T(g)$ is a map of complexes. Since $S(g)$ is entirely independent 
	of choices, the maps on cohomology induced by $T(g)$ are independent of choices as well.
	
	Given $f \colon \Gamma_j^q \to A_j$, consider the unique homogeneous cochain $F$ as above such that
	$$
		f(\gamma_1, \ldots, \gamma_q) = F(1,\gamma_1, \ldots, \gamma_1 \cdots \gamma_q),
	$$ 
	and recall that this induces an isomorphism between the inhomogeneous and homogeneous cochain complexes.
	The proposition now follows from the computation
	\begin{align*}
		T(g)F(1,\gamma_1, \ldots, \gamma_1 \cdots \gamma_q)
		&= \sum_{t=1}^v g_t F(1,\tau_{\sigma_{\gamma_1}^{-1}(t)}(\gamma_1), 
		\ldots, \tau_{\sigma_{\gamma_1}^{-1}(t)}(\gamma_1) \cdots \tau_{\sigma_{\gamma_q}^{-1}(t)}(\gamma_q))\\
		&= \sum_{t=1}^v g_tf(\tau_{\sigma_{\gamma_1}^{-1}(t)}(\gamma_1), 
		\ldots, \tau_{\sigma_{\gamma_1 \cdots \gamma_q}^{-1}(t)}(\gamma_q))\\
		&= \sum_{t=1}^v g_{\sigma_{\gamma_1 \cdots \gamma_q}(t)} f(\tau_{\sigma_{\gamma_2 \ldots \gamma_q}(t)}(\gamma_1),
		\ldots, \tau_t(\gamma_q))\\
		&= T(g)f(\gamma_1, \ldots, \gamma_q),
	\end{align*}
	as $\sigma = \sigma_{\gamma_1 \cdots \gamma_q}$ and $\mu_{t,w} =  \tau_{\sigma_{\gamma_{w+1} \cdots \gamma_q}(t)}(\gamma_w)$
	for $1 \le t \le v$ and $1 \le w \le q$.
\end{proof}

We have an alternative description of $T(g)$, following along similar lines to \cite[9.4(c)]{hida}.
For $g \in \Delta_{i,j}$, consider the operator 
$$
	\phi_g \colon H^q(\Gamma_j \cap g^{-1}\Gamma'_i g, A_j) \to H^q(g\Gamma_jg^{-1} \cap \Gamma'_i,A_i)
$$
induced by the map taking a cochain $f \colon (\Gamma_j \cap g^{-1}\Gamma'_i g)^q \to A_j$ to the cochain
$\phi_g(f)$ satisfying 
$$
	\phi_g(f)(\gamma_1, \ldots, \gamma_q) = g f(g^{-1}\gamma_1g, \ldots, g^{-1}\gamma_qg)
$$
for $\gamma_1,\ldots,\gamma_q \in g\Gamma_jg^{-1} \cap \Gamma'_i$.

\begin{proposition} \label{Hecke_corr}
	The Hecke operator $T(g)$ for $g \in \Delta_{i,j}$ equals the composition
	$$
		H^q(\Gamma_j,A_j) \xrightarrow{\res} H^q(\Gamma_j \cap g^{-1}\Gamma'_i g, A_j) \xrightarrow{\phi_g}
		H^q(g\Gamma_jg^{-1} \cap \Gamma'_i,A_i) \xrightarrow{\cor} H^q(\Gamma'_i,A_i),
	$$
	where $\res$ and $\cor$ denote restriction and corestriction, respectively.
\end{proposition}

\begin{proof}
	Write
	$$
		\Gamma'_i = \coprod_{t=1}^v \nu_t(g\Gamma_jg^{-1} \cap \Gamma'_i).
	$$
	with $\nu_t \in \Gamma'_i$ for $1 \le t \le v$.
	Then
	$$
		\Gamma'_i g \Gamma_j = \coprod_{t=1}^v \nu_t (g\Gamma_jg^{-1} \cap \Gamma'_i)g \Gamma_j
		= \coprod_{t=1}^v \nu_t g  \Gamma_j,
	$$
	so setting $g_t = \nu_t g \in \Delta_{i,j}$, we have a coset decomposition as in \eqref{cosetreps}.
	For $\gamma \in \Gamma'_i$, let $\sigma_{\gamma} \in S_v$ and $\tau'_t(\gamma) \in 
	g\Gamma_jg^{-1} \cap \Gamma'_i$ for $1 \le t \le v$ be defined by $\gamma \nu_t = \nu_{\sigma_{\gamma}(t)} \tau'_t(\gamma)$.
	This implies that $\gamma g_t = g_{\sigma_{\gamma}(t)} \cdot  g^{-1} \tau'_t(\gamma) g$,
	which means that 
	$$
		\tau_t(\gamma) = g^{-1}\tau'_t(\gamma)g \in \Gamma_j \cap g^{-1}\Gamma'_ig,
	$$ 
	where $\tau_t(\gamma)$ is as in the proof of Proposition \ref{Hecke_action}.
	
	The corestriction map on homogeneous cochains sends
	$C \colon (g\Gamma_jg^{-1} \cap\Gamma'_i )^{q+1}\to A_i$ to
	$$
		\cor(C)(\gamma_1, \ldots, \gamma_{q+1}) = \sum_{t=1}^v \nu_t C(\tau'_{\sigma_{\gamma_1}^{-1}(t)}(\gamma_1), 
		\ldots, \tau'_{\sigma_{\gamma_{q+1}}^{-1}(t)}(\gamma_{q+1}))
	$$
	for $\gamma_1, \ldots, \gamma_{q+1} \in \Gamma'_i$
    (cf.~\cite[Section 1.5]{nsw}).
	Then, given a homogeneous cochain $F \colon \Gamma_j^{q+1} \to A_j$, we have
	\begin{align*}
		\cor(\phi_g(\res(F)))(\gamma_1, \ldots, \gamma_{q+1})
		&=  \sum_{t=1}^v \nu_t \phi_g(\res(F))(\tau'_{\sigma_{\gamma_1}^{-1}(t)}(\gamma_1), 
		\ldots, \tau'_{\sigma_{\gamma_{q+1}}^{-1}(t)}(\gamma_{q+1}))\\
		&= \sum_{t=1}^v g_t F(\tau_{\sigma_{\gamma_1}^{-1}(t)}(\gamma_1), 
		\ldots, \tau_{\sigma_{\gamma_{q+1}}^{-1}(t)}(\gamma_{q+1}))\\
		&= T(g)F(\gamma_1, \ldots, \gamma_{q+1}).
	\end{align*}
\end{proof}

We end this subsection by explaining how an element $y \in \tilde{\Delta}$ with a certain property gives rise to Hecke operators $T(g)$ on any $H^q(\Gamma_j,A_j)$.

\begin{proposition} \label{adele_gives_Hecke}
	Let $y \in \tilde{\Delta}$ be such that $\det(U' \cap yUy^{-1}) = R^{\times}$, and let $i, j \in J$ be such that $\det(x_i^{-1}yx_j)$ has trivial ideal class.
	Then there exists $g \in \Delta_{i,j}$ such that $U' x_igx_j^{-1}U = U'yU$, and the coset $\Gamma'_i g \Gamma_j$
	is independent of the choice of $g$. In particular, $T(g)$ provides a Hecke operator
	$$
		T_{\mb{A}}(y) \colon H^q(\Gamma_j,A_j) \to H^q(\Gamma'_i,A_i)
	$$
	depending only the double coset $U'yU$.
\end{proposition}

\begin{proof}
	By strong approximation, 
	$\det \colon G \backslash \mc{G} / V \to (\mb{A}_F^f)^{\times}/F^{\times}\det(V)$ is a bijection
	for any open compact subgroup $V$ of $\mc{G}$. Of course, we have $(\mb{A}_F^f)^{\times}/F^{\times}R^{\times} = \Cl(F)$, so
	if $\det(V) = R^{\times}$, then since $\det(x_i^{-1}yx_j)$ is trivial in $\Cl(F)$, there exist $g \in G$ and $v \in V$ 
	such that $g = x_i^{-1} y x_j v$. 
	In particular, taking such a $V$ contained in $U_j$, we have $g \in \tilde{\Delta}_{i,j} \cap G = \Delta_{i,j}$.
	Further taking $V = x_j^{-1}(y^{-1}U'y \cap U)x_j$ so that $v \in x_j^{-1}y^{-1}U'y x_j$, we then have
	$x_igx_j^{-1} \in U'y$, yielding the equality of double cosets. 
	
	Let $g \in \Delta_{i,j}$ be as in the statement, which tells us that it satisfies $U'_igU_j = U'_i x_i^{-1} y x_j U_j$.
	Set $W = U'_i \cap gU_jg^{-1}$, and note that $W = x_i^{-1}u^{-1}(U' \cap yUy^{-1})ux_i$ for $u \in U'$ such that $g \in x_i^{-1}u^{-1}yx_j U_j$, so
	$\det(W) = R^{\times}$ by assumption. 
	By strong approximation again, $G \backslash \mc{G} / U'_i \to G \backslash \mc{G}/ W$ is a bijection, so 
	$(G \cap U'_i)\backslash U_i' /W$ is a singleton, which is to say that $U'_i = \Gamma'_i W$.
	In particular, we have 
	$$
	 	U'_i gU_j = \Gamma'_i g (g^{-1}Wg) U_j = \Gamma'_i g U_j.
	$$ 
	It follows that $U'_i x_i^{-1} y x_j U'_j \cap G = \Gamma'_i g \Gamma_j$, and therefore the latter double coset is independent of $g$.
\end{proof}

\subsection{Hecke operators as correspondences} \label{Hecke_op_corr}

The description of Hecke operators given by Proposition \ref{Hecke_corr} allows for comparison with Hecke actions on locally symmetric spaces.  Let $A$ be a $\Delta$-module system indexed by a set $J = \{ f(r) \mid r \in I \}$ for a function $f \colon I \to I^n$ such that each $\mf{a}_r^{-1} \prod_{u=1}^n \mf{a}_{f(r)_u}$ is principal (e.g., $f(r) = (r,1,\ldots,1)$). In this subsection, we shall henceforth identify $J$ with $I$ so that $A_{f(r)}$ is denoted more simply by $A_r$, and similarly with other subscripts, such as on $\Gamma$.

Let $\mb{H}_{n,F}$ denote the symmetric space $\GL_n(F \otimes_{\Q} \R)/(F \otimes_{\Q} \R)^{\times}\mr{O}_n(F \otimes_{\Q} \R)$,
where $\mr{O}_n(F \otimes_{\Q} \R)$ agrees with the product of degree $n$ orthogonal and unitary groups of the completions of $F$ at its real and complex places, respectively.
Suppose that $\det(U) = R^{\times}$, in which case the locally symmetric space 
$$
	Y(U) = \GL_n(F) \backslash (\GL_n(\mb{A}_F^f) \times \mb{H}_{n,F}) / U,
$$ 
with $G = \GL_n(F)$ acting diagonally and $U$ acting on $\mc{G} = \GL_n(\mb{A}_F^f)$ on the right,
is a disjoint union of components homeomorphic to $Y_r = \Gamma_r \backslash \mb{H}_{n,F}$ for $r \in I$, with the inclusion of $Y_r$ in $Y = Y(U)$ given by $z \mapsto (x_r^{-1},z)$.

The coefficient system $\mc{A} = \mc{A}(U)$ on $Y$ given by
$\mc{A}_r = \mc{A}_r(U) = \Gamma_r \backslash (\mb{H}_{n,F} \times A_r)$ on $Y_r$ gives rise to a constructible sheaf on $Y$ 
(i.e., of its continuous sections) that we again give the notation $A$, with its restriction to $Y_r$ denoted $A_r$.

For each $r \in I$, we suppose that the orders of all torsion elements in $\Gamma_r$ act invertibly on $A_r$ and that the
scalar elements in $\Gamma_r$ act trivially on $A_r$.
We then have isomorphisms
\begin{equation} \label{isom_cohom}
	\bigoplus_{r \in I} H^q(\Gamma_r,A_r) \cong H^q(Y,A)
\end{equation}
for $q \ge 0$. These can be described as the sum over $r$ of the compositions
$$
	H^q(\Gamma_r,A_r) \xrightarrow{\sim} H^q(\Gamma_r,H^0(\mb{H}_{n,F},A_r)) \xrightarrow{\sim} H^q(Y_r,A_r)
$$
the first map coming from the canonical isomorphism $A_r \xrightarrow{\sim} H^0(\mb{H}_{n,F},A_r)$.
The inverse of the second map is induced by the chain map taking an $A_r$-valued simplicial $q$-cochain $F$ on $Y_r$ to the cochain $f \colon \Gamma_r^q \to A_r$ such that $f(\gamma_1,\ldots,\gamma_q)$ is the 
function taking $x \in \mb{H}_{n,F}$ to the 
value of $F$ on the image in $Y_r$ of the geodesic $q$-simplex on $\mb{H}_{n,F}$ with vertices $x, \gamma_1 x, \ldots, \gamma_1 \ldots \gamma_q x$.

We define actions of Hecke operators on the left of \eqref{isom_cohom}. We focus here on the case that $U = U'$ for simplicity of 
notation, but distinct $U$ and $U'$ can be treated by making the necessary changes of notation.
For $g \in \GL_n(F)$, set $Y_{r,s}^g = (\Gamma_r \cap g\Gamma_sg^{-1}) \backslash \mb{H}_{n,F}$,
and consider the coefficient system
$$
	\mc{A}_{r,s}^g =  (\Gamma_r \cap g\Gamma_s g^{-1}) \backslash (\mb{H}_{n,F} \times A_r).
$$
Any $g \in \Delta_{r,s}$ defines a homeomorphism $g \colon Y_{s,r}^{g^{-1}} \to Y_{r,s}^g$ 
via left multiplication on $\mb{H}_{n,F}$. 
Together with the map $g \colon A_s \to A_r$, this induces a map
$\phi_g \colon H^q(Y_{s,r}^{g^{-1}},A_s) \to H^q(Y_{r,s}^g,A_r)$ on sheaf cohomology. 

\begin{definition}
For $g \in \Delta_{r,s}$, 
we define the \emph{Hecke operator}
$T(g) \colon H^q(Y_s,A_s) \to H^q(Y_r,A_r)$ of $g$ by
$$
	H^q(Y_s,A_s) \xrightarrow{\res} H^q(Y_{s,r}^{g^{-1}},A_s) \xrightarrow{\phi_g} H^q(Y_{r,s}^g,A_r)
	\xrightarrow{\cor} H^q(Y_r,A_r),
$$
where here the maps $\res$ and $\cor$ are restriction and trace, respectively. 
\end{definition}

The following is then a consequence
of Proposition \ref{Hecke_corr}.

\begin{proposition} \label{Hecke_gp_top}
	For $g \in \Delta_{r,s}$, the diagram
	$$
	\begin{tikzcd}
		H^q(\Gamma_s,A_s) \arrow{r}{T(g)}  \arrow{d}{\wr} & H^q(\Gamma_r,A_r) \arrow{d}{\wr} \\
		H^q(Y_s,A_s) \arrow{r}{T(g)} & H^q(Y_r,A_r) 
	\end{tikzcd}
	$$
	commutes.
\end{proposition}

In particular, it follows from Proposition \ref{Hecke_action} that the Hecke
operator $T(g) \colon H^q(Y_s,A_s) \to H^q(Y_r,A_r)$ is independent of the
choice of representative of $\Gamma_r g \Gamma_s$.

We can also view these $T(g)$ for $g \in \Delta_{r,s}$ as coming from a single adelic operator.
Being a bit unrigorous in our notation for motivational purposes, it is useful to note that $\phi_g(x_s^{-1},z,a_s) \in \mc{A}_{r,s}^g$ 
may be viewed in terms of representatives of double cosets as
$$
	 (x_r^{-1},gz,ga_s) = (g^{-1}x_r^{-1},z,a_s) = (x_s^{-1} (x_sg^{-1}x_r^{-1}),z,a_s) = (x_s^{-1},z,a_s) \cdot (x_rgx_s^{-1})^{-1},
$$
where the latter multiplication is that of $\GL_n(\mb{A}_F^f)$ on the right. We remark that $x_r g x_s^{-1} \in \tilde{\Delta}$
by definition of $\Delta_{r,s}$. 

For $y \in \mc{G}$, let $U^y = U \cap yUy^{-1}$. 
Right multiplication by $y^{-1}$ defines a homeomorphism $y^{-1} \colon Y(U^{y^{-1}}) \to Y(U^y)$.
It also gives rise to a map on cohomology, as we now explain. 

\begin{proposition} \label{compare_mult_maps}
	For $y \in \tilde{\Delta}$ such that $\det(U^y) = R^{\times}$, there exists a unique homomorphism
	$$
		\psi_y \colon H^q(Y(U^{y^{-1}}),A) \to H^q(Y(U^y),A)
	$$  
	that, for each pair $(r,s) \in I^2$ such that the ideal attached to $\det(x_r^{-1}yx_s)$ is principal, restricts to
	$\phi_g$ for any $g \in \Delta_{r,s}$ such that 
	\begin{equation} \label{left_coset}
		U^y x_r g x_s^{-1}  = U^y y. 
	\end{equation}
\end{proposition}

\begin{proof}
	The existence of $g \in G$ satisfying \eqref{left_coset}
	is by strong approximation. Any such $g$ lies in $\Delta_{r,s}$, which we can
    see by rearranging \eqref{left_coset} as
    \begin{equation} \label{rearrange}
        g \in x_r^{-1}U^yx_r \cdot x_r^{-1}yx_s \subseteq U_r \cdot \tilde{\Delta}_{r,s}
        = \tilde{\Delta}_{r,s}.
    \end{equation}
    The condition \eqref{left_coset} implies the two equalities $
    Uy^{-1}x_r = U x_s g^{-1}$ and  $U yx_s = U x_rg$.
	The former gives
	\begin{equation} \label{U_Gamma}
    x_r^{-1} U^y x_r \cap G = \Gamma_r \cap g \Gamma_s g^{-1},
    \end{equation}
	and the latter tells us that 
    $x_s^{-1} U^{y^{-1}} x_s \cap G = \Gamma_s \cap g^{-1} \Gamma_r g$.
	Thus, $H^q(Y(U^{y^{-1}}),A)$ has a direct summand isomorphic to $H^q(Y_{s,r}^{g^{-1}},A_s)$, and $H^q(Y(U^y),A)$ has one isomorphic
	to $H^q(Y_{r,s}^g,A_r)$. The groups $H^q(Y(U^{y^{-1}}),A)$ and $H^q(Y(U^y),A)$
    are the direct sums of these summands over $r \in I$, so
    $\psi_y$ exists. As for uniqueness, \eqref{rearrange} and
    \eqref{U_Gamma} together show that
    if $g, g' \in \Delta_{r,s}$ both satisfy \eqref{left_coset}, then
	$g'g^{-1} \in \Gamma_r \cap g\Gamma_s g^{-1}$, 
    so $\phi_g = \phi_{g'}$.
\end{proof}

\begin{definition}
For $q \ge 0$, we define the \emph{adelic Hecke operator} 
$$
	T_{\mb{A}}(y) \colon H^q(Y,A) \to H^q(Y,A)
$$
of $y \in \tilde{\Delta}$ such that $\det(U^y) = R^{\times}$
as the composition of pullback to $Y(U^{y^{-1}})$, the map $\psi_y$, and the trace from $Y(U^y)$ to $Y$.
\end{definition}

The following is a corollary of Proposition \ref{compare_mult_maps} with the observation from Proposition \ref{adele_gives_Hecke} that the coset
$\Gamma_r g \Gamma_s$ depends only on $UyU$.

\begin{proposition} \label{Hecke_subsp}
	For $y \in \tilde{\Delta}$ for such that $\det(U^y) = R^{\times}$
	and $s \in I$, let $r \in I$ be unique such that the ideal attached to $\det(x_r^{-1}yx_s)$ is principal.
	Let $g \in \Delta_{r,s}$ be any element such that
	$$
		U x_r g x_s^{-1} U = UyU.
	$$
	Then the following diagram commutes 
	$$
	\begin{tikzcd}
		H^q(Y_s,A_s) \arrow{r}{T(g)} \arrow[hook]{d}{} & H^q(Y_r,A_r) \arrow[hook]{d}{} \\
		H^q(Y,A) \arrow{r}{T_{\mb{A}}(y)} & H^q(Y,A).
	\end{tikzcd}
	$$
\end{proposition}

Putting Propositions \ref{Hecke_gp_top} and \ref{Hecke_subsp} together with Proposition \ref{adele_gives_Hecke}, we see that the constructions of $T_{\mb{A}}(y)$ on the left and right-hand sides of \eqref{isom_cohom} coincide.

\subsection{Hecke operators attached to ideals} \label{Hecke_ideal}

Let us suppose now that $A = (A_i)_{i \in I^n}$ is a $\Delta$-module system for $I^n$. We impose two additional conditions on our open subgroup $U$ of $\GL_n(R)$: 
\begin{enumerate}
    \item[i.] $U$ is normalized by diagonal matrices in $\GL_n(R)$, and 
    \item[ii.] $U$ contains the subgroup of $G$ consisting of diagonal matrices in $\GL_n(\mc{O})$.
\end{enumerate}
Condition (i) implies that the group $\Gamma_i$ is independent of the choice of the $\alpha_r$ representing $\mf{a}_r$ for $r \in I$. Condition (ii)
allows us to make the following definition, independent of choice.

\begin{definition} \label{Hecke_ops_ideals}
	Let $\mf{n}_1, \ldots, \mf{n}_n$ be nonzero ideals of $\mc{O}$. 
	For each $i \in I^n$, let $j \in I^n$
	be unique such that $\mf{a}_{i_u,j_u}\mf{n}_u$ is principal for all $1 \le u \le n$. Let $\eta_u$ 
	be a generator of $\mf{a}_{i_u,j_u}\mf{n}_u$ for each $u$, and
    set $g = \diag(\eta_1, \ldots, \eta_n)$.
	If $g \in \Delta_{i,j}$, then we define the \emph{Hecke operator} 
	$$
		T(\mf{n}_1, \ldots, \mf{n}_n) \colon H^q(\Gamma_j,A_j) \to H^q(\Gamma_i,A_i)
	$$ 
	for $(\mf{n}_1, \ldots, \mf{n}_n)$ as  $T(g)$.
\end{definition}

We will be particularly interested in the following operators.

\begin{definition}
	For a nonzero ideal $\mf{n}$ of $\mc{O}$ and $1 \le u \le n$, 
    we set 
	$$
		T_{\mf{n}}^{(u)} = T(\mf{n},\ldots,\mf{n},1,\ldots,1)
	$$
    when the latter operator exists, the expression for which 
    contains $u$ copies of $\mf{n}$.
    We write $T_{\mf{n}} = T_{\mf{n}}^{(1)}$ and $[\mf{n}]^*
    = T_{\mf{n}}^{(n)}$.
\end{definition}

We have the following simple lemma regarding the operators $[\mf{n}]^*$.

\begin{lemma} 
	Let $i, j \in I^n$ be such that $\mf{a}_{i_u,j_u}\mf{n}$ is principal for all $1 \le u \le n$. If $[\mf{n}]^*$ exists, then its double coset decomposition is
	$$ 
		\Gamma_i \smatrix{ \eta_1 \\ &\ddots \\ && \eta_n } \Gamma_j
		= \smatrix{ \eta_1 \\ &\ddots \\ && \eta_n } \Gamma_j,
	$$ 
	where $\eta_u$ is a generator of $\mf{a}_{i_u,j_u}\mf{n}$ for each $u$.
\end{lemma}

\begin{proof}
    Let $h = \diag(\eta_1, \ldots, \eta_n)$.
    The id\`{e}le $\eta_u^{-1}\alpha_{i_u}^{-1}\alpha_{i_v}\eta_v$ has associated ideal generating $(\mf{a}_{i_u,j_u}\mf{n})^{-1}\mf{a}_{i_u,i_v}\mf{a}_{i_v,j_v}\mf{n} = \mf{a}_{j_u,j_v}$
    for $u, v \in \{1,\ldots,n\}$, so it is a multiple of
    $\alpha_{j_u}^{-1}\alpha_{j_v}$ by a unit in $R$. 
    Therefore, $h^{-1} U_i h = h^{-1}x_i^{-1} U x_ih$ 
    is conjugate to $U_j = x_j^{-1}Ux_j$ by a diagonal matrix in $\GL_n(R)$, so it is equal to $U_j$ by assumption on $U$.
    Then $h^{-1}\Gamma_i h = \Gamma_j$, so we are done.
\end{proof} 

For an ideal $\mf{a}$ of $\mc{O}$,
let $F_{\mf{a}}^{\times}$ denote the group of finite id\`{e}les
of $F$ that are $1$ at all primes not dividing $\mf{a}$.
For $r \in I$, we now choose the id\`ele $\alpha_r \in R$ with associated ideal $\mf{a}_r$ to lie in $F_{\mf{a}_r}^{\times}$.

Let $\mc{N}$ be a nonzero ideal of $\mc{O}$ that is prime to $\mf{a}_r$ for $r \in I$.
Let $U(\mc{N})$ denote the open subgroup 
of $\GL_n(R)$ consisting of matrices with image in $\GL_n(\mc{O}/\mc{N})$
contained in the image of the diagonal matrices in $\GL_n(\mc{O})$. 
If $U$ contains $U(\mc{N})$,
we refer to the largest ideal $\mc{M}$ such that $U$ contains $U(\mc{M})$ as the \emph{level} of $U$. 
We suppose that $U$ has level $\mc{N}$. 

Let us define an adelic analogue of Definition \ref{Hecke_ops_ideals} that works for more pairs of elements of
$I^n$ but provides slightly different
Hecke operators in cases where both are defined. 

\begin{definition} \label{more_gen_Hecke}
	Let $\mf{n}_1, \ldots, \mf{n}_n$ be nonzero ideals of $\mc{O}$. Let $\nu_u \in F_{\mf{n}_u}^{\times}$ have associated ideal $\mf{n}_u$.
	For all primes $\mf{p}$ of $\mc{O}$ dividing $\mf{n}_u+\mc{N}$, we assume that $U$ contains all diagonal matrices in 
	$\GL_n(R_{\mf{p}})$ with $v$th entry $1$ for all $v \neq u$.
	For $i, j \in I^n$ such that $\prod_{u=1}^n \mf{a}_{i_u,j_u}\mf{n}_u$ is principal,
	we define
	$$
		T_{i,j}(\mf{n}_1, \ldots, \mf{n}_n) \colon H^q(\Gamma_j,A_j) \to H^q(\Gamma_i,A_i)
	$$ 
	to be the restriction of $T_{\mb{A}}(\diag(\nu_1, \ldots, \nu_n))$.
\end{definition}

If we suppose that the orders of torsion elements in $\Gamma_i$ acts invertibly and scalar elements act trivially on $A_i$ for each $i \in I^n$, then we can also define $T_{i,j}(\mf{n}_1, \ldots, \mf{n}_n) \colon H^q(Y_j,A_j) \to H^q(Y_i,A_i)$ as the restriction of $T_{\mb{A}}(\diag(\nu_1, \ldots, \nu_n))$, and these operators are compatible in the sense of Proposition \ref{Hecke_gp_top}.

\begin{remark} \label{compare_Hecke}
    Suppose that the ideals $\mf{n}_1, \ldots, \mf{n}_n$ are prime to $\mc{N}$.
	Taking $i, j \in I^n$ such that $\mf{a}_{i_u,j_u}\mf{n}_u$ is principal
    for all $1 \le u \le n$, the operator
	$T_{i,j}(\mf{n}_1, \ldots, \mf{n}_n)$ is given by $T(g)$ with $g$ diagonal and congruent to the identity matrix modulo $\mc{N}$, whereas
	$T(\mf{n}_1, \ldots, \mf{n}_n)$ is $T(h)$ for $h$ the diagonal matrix with $u$th entry a generator of $\mf{a}_{i_u,j_u}\mf{n}_u$,
	which may not be $1$ modulo $\mc{N}$.
\end{remark}

We are concerned in this paper in a situation for which $n = 2$.
For the rest of this subsection, take $\tilde{\Delta}$ to be the submonoid $\tilde{\Delta}_0(\mc{N})$ of $\GL_2(\mb{A}_F^f) \cap M_2(R)$ consisting of elements
with bottom row $(0,z)$ modulo $\mc{N}$, where $z$ is prime to $\mc{N}$.
We shall need the following two open subgroups of $\GL_2(R)$ of level $\mc{N}$ contained in $\tilde{\Delta}_0(\mc{N})$.
The first is the group $U_0(\mc{N})$
consisting of matrices with second row congruent to $(0,a)$ modulo $\mc{N}$
for some $a \in R^{\times}$. The second is the subgroup 
$U_1(\mc{N})$ consisting of matrices in $U_0(\mc{N})$ with $(2,2)$-entry congruent to an element of $\mc{O}^{\times}$
modulo $\mc{N}$. We set $\Gamma_*(\mc{N})_i = U_*(\mc{N})_i \cap G$ for $* \in \{0,1\}$. For the remainder of this subsection, we take
$n = 2$ and suppose that $U$ contains $U_1(\mc{N})$.

\begin{definition} \label{diamondop}
	Let $\mf{d}$ be an ideal of $\mc{O}$
	prime to $\mc{N}$. Let $i, j \in I^2$ be such that $\mf{a}_{i_1,j_1}\mf{d}^{-1}$ and $\mf{a}_{i_2,j_2}\mf{d}$
	are principal. Let $\lambda$ be a generator of the latter ideal. 
	For any matrix $\delta \in \Delta_0(\mc{N})_{i,j}$ 
	with 
	$$
		\mf{a}_{i_1,j_1}\mf{a}_{i_2,j_2} = \det(\delta)\mc{O}
	$$
	and bottom right entry $\lambda$ modulo $\mc{N}$, we define the 
	(adjoint) \emph{diamond operator} $\langle \mf{d} \rangle^*$ attached to $\mf{d}$
	by 
	$$
		\langle \mf{d} \rangle^* = T(\delta) \colon H^q(\Gamma_j,A_j) \to H^q(\Gamma_i,A_i)
	$$ 
	for $q \ge 0$. Note that such a matrix $\delta$ exists by the strong approximation theorem.
\end{definition}

The reader may verify the following.

\begin{lemma} \label{diamond_lem}
	We maintain the notation of Definition \ref{diamondop}.
	\begin{enumerate}
		\item[a.] We have $\Gamma_i \delta \Gamma_j = \delta \Gamma_j$, and $\delta \Gamma_j$ depends only upon the image 
		of $\mf{d}$ in $\Cl_{\mc{N}}(F)$.
		\item[b.] We have $\delta^{-1} \in \Delta_0(\mc{N})_{j,i}$, and $T(\delta^{-1}) = \langle \mf{e} \rangle^*$ for any ideal $\mf{e}$ of 
		$\mc{O}$ inverse to $\mf{d}$ in $\Cl_{\mc{N}}(F)$.
	\end{enumerate}
\end{lemma}

The following amounts to a special case of Remark \ref{compare_Hecke}.

\begin{lemma} \label{different_Hecke}
	Let $\mf{n}$ be a nonzero ideal of $\mc{O}$. Let $q \ge 0$.
	\begin{enumerate}
		\item[a.] Suppose that $\mf{n}$ is prime to $\mc{N}$.
        For $i, j, k \in I^2$ such that $\mf{a}_{i_1,k_1}\mf{n}$, $\mf{a}_{i_2,k_2}\mf{n}$, and 
		$\mf{a}_{i_1,j_1}\mf{n}^{-1}$ are principal and $j_2 = k_2$, we have
		$$
			[\mf{n}]^* = \langle \mf{n} \rangle^* \circ T_{j,k}(\mf{n},\mf{n}) \colon H^q(\Gamma_k,A_k) \to H^q(\Gamma_i,A_i).
		$$
		\item[b.] For $i, j \in I^2$ such that $\mf{a}_{i_1,j_1}\mf{n}$ is principal and $i_2 = j_2$, we have
		$$
			T_{\mf{n}} = T_{i,j}(\mf{n},1) \colon H^q(\Gamma_j,A_j) \to H^q(\Gamma_i,A_i).
		$$
	\end{enumerate}
\end{lemma}

\begin{proof}
	Let $\eta_u$ generate $\mf{a}_{i_u,k_u}\mf{n}$ for $u \in \{1,2\}$, and let $g = \smatrix{\eta_1\\&\eta_2}$.
	Let $\nu \in R^{\times}$ with $\nu \equiv \eta_2 \bmod \mc{N}$.
	Let $\tilde{\delta} \in \tilde{\Delta}_{i,j}$ be diagonal with $(1,1)$-entry $\alpha_{j_1}\alpha_{i_1}^{-1}\nu^{-1}$ and $(2,2)$-entry
	$\alpha_{j_2}\alpha_{i_2}^{-1}\nu$. 
	Then $U_i \tilde{\delta} U_j = U_i \delta U_j$ for $\delta \in \Delta_{i,j}$ with $T(\delta) = \langle \mf{n} \rangle^*$,
	with both double cosets equal to a single left coset.
	 
	The matrix $\tilde{\delta}^{-1}g \in \tilde{\Delta}_{j,k}$ has first and second diagonal entries with associated ideals
	$\mf{a}_{j_1,k_1}\mf{n}$ and $\mf{n}$, respectively, with the second being $1$ modulo $\mc{N}$.
	Since $U$ contains $U_1(\mc{N})$, it contains all diagonal matrices with $(1,1)$-entry in $R^{\times}$
	and $(2,2)$-entry in $\mc{O}^{\times}$ modulo $\mc{N}$. It follows that
	the double coset $U_j \delta^{-1}g U_k = U_j \tilde{\delta}^{-1}g U_k$ is  equal to the double coset of
	an element of $\tilde{\Delta}_{j,k}$ that yields the operator $T_{j,k}(\mf{n},\mf{n}) = T(\delta^{-1}g)$ of 
	Definition \ref{more_gen_Hecke}. This yields part (a).
	
	Part (b) is essentially immediate, again using the fact that $U$ contains $U_1(\mc{N})$.
\end{proof}

\subsection{$\Delta$-module systems with pushforwards}

In our applications, $\Delta$-module systems are direct limits 
of motivic cohomology groups of open subschemes of commutative group schemes, with the elements of $\Delta$ (or more precisely, each $\Delta_{i,j}$) providing isogenies between these group schemes. The $\Delta$-action on $A$ is
then one of pullback. However, these motivic cohomology groups also come equipped with pushforwards by elements of $\Delta$, which we shall have occasion to employ. To fit this
into our abstract framework, let us define the notion of pushfoward maps on
a $\Delta$-module system. To distinguish these from the maps in a $\Delta$-module
system, we use the standard notation for pullbacks for the latter maps, writing $g \colon A_j \to A_i$ for $g \in \Delta_{i,j}$ as $g^*$ in this subsection.

\begin{definition}
    Let $A$ be a $\Delta$-module system for $J \subseteq I^n$, and 
    let $\tilde{\Delta}'$ be a submonoid of $\tilde{\Delta}$.
    A system of $\Delta'$-pushforwards on $A$
    is a collection of \emph{pushforward maps} 
    $g_* \colon A_i \to A_j$ with $g \in \Delta'_{i,j}$ for $i,j \in J$ such that
    \begin{enumerate}
        \item[i.] $h_* \circ g_* = (gh)_*$ if $h \in \Delta'_{j,k}$,
        \item[ii.] $(1_n)_*$ is the identity on $A_i$, and
        \item[iii.] $g_* \circ g^* = (\det(g)1_n)_*$.
    \end{enumerate}
\end{definition}

We have included (iii) as a standard compatibility of pushforwards and pullbacks, but we do not actually use it in this paper. In fact, we eschew presenting any semblance of a general theory and focus only on the single definition needed in this work.

Let us suppose that $U$ is an open subgroup of $\GL_n(R)$ satisfying conditions (i) and (ii) of Section \ref{Hecke_ideal}. 
Let $A$ be a $\Delta$-module system with a system of pushforwards for the submonoid of diagonal matrices in $\tilde{\Delta}$.

\begin{definition}
    Given $i,j \in I^n$ and a nonzero ideal $\mf{b}$ of $\mc{O}$, we write
    $i \sim_{\mf{b}} j$ to denote that $\mf{a}_{i_u,j_u}\mf{b}$
    is principal for all $1 \le u \le n$.
\end{definition}

\begin{definition} \label{isom_cplx}
    Let $\mf{b}$ be a nonzero ideal of $\mc{O}$. Let $i,j \in J$ be such
    that $i \sim_{\mf{b}} j$.
    Let $\rho \in \Delta_{i,j}$ be diagonal with $u$th entry generating
    $\mf{a}_{i_u,j_u}\mf{b}$. For 
    $f \colon \Gamma_i^q \to A_i$
	we define $[\mf{b}]_*(f) \colon \Gamma_j^q \to A_j$ on $\mu \in \Gamma_j^q$
	by
	$$
		[\mf{b}]_*(f)(\mu) = \rho_*f(\rho \mu \rho^{-1}).
	$$
\end{definition}

If $\mf{b} = (b)$ is principal, then we also write $[b]_*$ for $[\mf{b}]_*$.
As $\Gamma_j$ contains the diagonal matrices in $\GL_n(\mc{O})$, the following is easily verified.

\begin{lemma} \label{pushforward_ideal}
    Let $i,j \in J$ and $\mf{b}$ be a nonzero ideal of $\mc{O}$ such that
    $i \sim_{\mf{b}} j$.
	Suppose that every choice of $\rho_*$ as in Definition \ref{isom_cplx} has the property that 
	\begin{equation} \label{gamma_commute}
		\rho_*\gamma^* = (\rho^{-1}\gamma\rho)^*\rho_*
	\end{equation}
	for all $\gamma \in \Gamma_i$.
	Then $[\mf{b}]_*$ defines a homomorphism of chain complexes,
    and it depends upon the choice of $\rho$ only up to chain homotopy. Given a homomorphism $\phi \colon A \to B$ of
    $\Delta$-module systems, we have $[\mf{b}]_* \circ \phi_i = \phi_j \circ
    [\mf{b}]_*$ for any fixed choice of $\rho$.
\end{lemma}

\section{$\GL_2$-cocycles for CM elliptic curves}

\subsection{CM elliptic curves} \label{CM_EC}

Keeping the notation of the last section, we now let $F$ be an imaginary quadratic field. 
All number fields will be considered as subfields of the algebraic numbers in $\C$.
There are $h = |\Cl(F)|$ isomorphism classes of elliptic curves over $\C$ with CM by $\mc{O}$. 
Each has a representative defined over the Hilbert class field $H$ of $F$. Recall that $I = \{1, \ldots, h\}$. 

We give a quick proof of the following known lemma.

\begin{lemma} \label{isog_field_of_def}
    Suppose that $L$ is a finite extension of $F$.
    Let $E$ and $E'$ be $L$-isogenous elliptic curves with CM by $\mc{O}$ defined 
    over $L$. Then all isogenies from $E$ to $E'$ are defined
    over $L$.
\end{lemma}

\begin{proof}
    By \cite[Theorem II.2.2]{silverman},
    this is true for $E=E'$. 
    Thus $\Hom_L(E,E')$ is a nonzero
    $\mc{O}$-submodule of $\Hom(E,E')$, hence of finite index.
    For $f \in \Hom(E,E')$, let $m \ge 1$ be such that $m f = f \circ m$ is 
    defined over $L$. Since multiplication by $m$
    is also defined over $L$, 
    we have $f \circ m = f^{\sigma} \circ m$ for any automorphism of $\C$ fixing $L$.
    Then $m(f-f^{\sigma}) = 0$,
    so $f = f^{\sigma}$.
\end{proof}

Let us fix a nonzero ideal $\mf{f}$ of $\mc{O}$ prime to all $\mf{a}_r$ for $r \in I$ 
such that $\mc{O}^{\times} \to (\mc{O}/\mf{f})^{\times}$ is injective. 
That is, we suppose that no nontrivial root of unity in $\mc{O}$ is $1$ modulo $\mf{f}$.

As described in the discussion of \cite[1.4]{deshalit}, there exists an elliptic curve over $L = F(\mf{f})$ with CM by $\mc{O}$ such
that its torsion points are all defined over an abelian extension of $F$. In fact, we may choose such a curve $E_1$
so that $E_1(\C) \cong \C/\mc{O}$ and fix such an analytic isomorphism.

Let $\mc{N}$ be a nonzero ideal of $\mc{O}$ also prime to $\mf{a}_r$ for all $r \in I$.
Let $\mc{R}$ denote the Artin map for $F(\mf{f} \cap \mc{N})/F$, and let $\sigma_r = \mc{R}(\mf{a}_r^{-1})$ for 
$r \in I$. Set $E_r = E_1^{\sigma_r}$.  By Lemma \ref{isog_field_of_def}, all isogenies between the curves $E_r$ are defined over $L$ as well. In fact, by CM theory (cf. \cite[1.5]{deshalit}), there is a unique $L$-isogeny $\psi_r \colon E_r \to E_1$ with kernel $E_r[\mf{a}_r]$ that agrees with $\sigma_r^{-1}$ on prime-to-$\mf{a}_r$-torsion. The identification 
of $E_1$ with $E_r/E_r[\mf{a}_r]$ gives rise to an analytic isomorphism $E_r(\C) \cong \C/\mf{a}_r$.
In turn, this supplies an isomorphism $\Hom_L(E_r,E_1) \cong \mf{a}_r^{-1}$.
For $r, s \in I$, since every analytic map $\C/\mf{a}_r \to \C/\mf{a}_s$ preserving $0$ corresponds to an isogeny which is necessarily defined over $L$, we then have identifications
\begin{equation} \label{isogeny_gp}
	\Hom_L(E_r,E_s) \cong \mf{a}_{r,s}.
\end{equation}

Now, let us turn to the $\mc{N}$-torsion on these elliptic curves.
For $\alpha \in F^{\times}$ prime to $\mathcal{N}$ and $r \in I$, let $[\alpha]_r \colon E_r[\mathcal{N}] \rightarrow E_r[\mathcal{N}]$ be the isomorphism given by multiplication by any element in $\mc{O}$ congruent to $\alpha$ modulo $\mathcal{N}$.
Let us set $\sigma_{r,s} = \sigma_s\sigma_r^{-1} = \mc{R}(\mf{a}_{s,r})$ for brevity.

\begin{proposition} \label{alpha_sigma}
	Let $d \in \mf{a}_{r,s}$ be prime to $\mc{N}$. We have
	\[
    	\sigma_{r,s} = d \circ [d]_r^{-1} = [d]_s^{-1} \circ d
	\]
	as group homomorphisms $E_r[\mathcal{N}] \rightarrow E_s[\mathcal{N}]$.
\end{proposition}

\begin{proof}
	It suffices to prove this in the case $s = 1$. Observe that the quantity $d \circ [d]_r^{-1} = [d]_1^{-1} \circ d$ 
	is independent of the choice of $d \in \mf{a}_r^{-1}$.
	The isogeny $\psi_r \colon E_r \to E_1$ is identified with $1 \in \mf{a}_r^{-1}$,
	and clearly $1 = 1 \circ [1]_r^{-1} \colon E_r[\mc{N}] \to E_1[\mc{N}]$. Since $\psi_r$ equals $\sigma_r^{-1}$
	on $E_r[\mc{N}]$, we are done.
\end{proof}

Let $t \in I$ be such that $E_t(\C) \cong \C/\mc{N}$. A fixed generator of $\mc{N}\mf{a}_t^{-1}$ (unique up to unit)
provides an isomorphism $\C/\mf{a}_t \to \C/\mc{N}$ of elliptic curves, and this gives an isomorphism $E_t(\C) \cong
\C/\mc{N}$. We let $Q$ be the primitive $\mc{N}$-torsion point of $E_t$ corresponding to $1 \in \C/\mc{N}$. Finally, we set $P_r = 
\sigma_{t,r}(Q)
 \in E_r[\mc{N}]$ for all $r \in I$. 
Proposition \ref{alpha_sigma} may then be rephrased as follows.

\begin{corollary} \label{action_on_pts}
	Let $d \in \mf{a}_{r,s}$ be prime to $\mc{N}$.
	Then
	$$
 		P_s = [d]_s^{-1} (d \cdot P_r).
	$$
\end{corollary}

We will also have use of the following.

\begin{lemma} \label{action_on_pts_single_curve}
	Let $d \in \mc{O}$ be prime to $\mc{N}$ with $d \equiv 1 \bmod \mf{f}$. For $r \in I$, we have
	$$
		\mc{R}(d)(P_r) = [d]_r(P_r).
	$$
\end{lemma}

\begin{proof}
	By CM theory (see \cite[Proposition 1.5]{deshalit}), we know that there exists an isogeny $\lambda \colon E_r \to E_r$
	with kernel $E_r[d]$ that agrees with $\mc{R}(d)$ on prime-to-$d$-torsion in $E_r$, and $\mc{R}(d)$ fixes $E_r[\mf{f}]$.
	The only endomorphism of $E_r$ with kernel $E_r[d]$ and fixing $E_r[\mf{f}]$ is multiplication by $d$, since there
	are no nontrivial units in $\mc{O}^{\times}$ that are $1$ modulo $\mf{f}$.
\end{proof}

\subsection{Motivic complexes of products of elliptic curves}

Let $\mc{R}$ be $\Z$ or an order in the ring of integers of an imaginary quadratic field.
For $j \in \{1,2\}$, let $A_j$ be an elliptic curve with endomorphism ring $\mc{R}$ over a characteristic $0$ field $L$, and set $\mc{A} = A_1 \times_L A_2$. As in \cite[(2.5)]{sv}, we have a complex $\ms{K}(\mc{A})$ in homological degrees $2$ to $0$:
$$
	K_2(L(\mc{A})) \xrightarrow{\partial} \bigoplus_D K_1(L(D)) \xrightarrow{\partial} \bigoplus_x K_0(L(x)),
$$
where $K_d$ denotes the $d$th $K$-group,
with the sums in degree $d$ taken over the irreducible $L$-cycles of dimension $d$ in $\mc{A}$. 
The homology of this complex is given by 
$$
	H_d(\ms{K}(\mc{A})) \cong H^{4-d}(\mc{A},\Z(2)).
$$

We have trace (i.e., pushforward) maps $[\alpha]_*$ for multiplication by elements
$\alpha \in \mc{R} - \{0\}$ on these groups. As such, we introduce the following notation.
Let $\Z' = \Z[\frac{1}{30}]$ throughout. For an abelian group $B$, we set $B_{\Z'} = B \otimes \Z'$.
For an abelian group $M$ with a multiplicative $(\mc{R} - \{0\})$-action, we use $M^{(0)}$ to denote the subgroup of elements of $M_{\Z'}$ 
that are fixed under all elements of $\mc{R}$ prime to some nonzero element. We refer to $M^{(0)}$ as the \emph{trace-fixed part} of $M$.

\begin{lemma} \label{K0torsion}
The trace-fixed part of $\ms{K}_0(\mc{A})$ satisfies
$$
	\ms{K}_0(\mc{A})^{(0)} \cong \varinjlim_{\mf{n}} H^0(\mc{A}[\mf{n}],\Z)^{(0)},
$$ 
where $\mf{n}$ runs over the nonzero ideals of $\mc{R}$.
In fact, the class of $\mc{A}[\mf{n}]$ is trace-fixed and is a sum of distinct trace-fixed classes that generate
$H^0(\mc{A}[\mf{n}],\Z)^{(0)}$.
\end{lemma}

\begin{proof}
	Any irreducible zero-cycle containing a non-torsion point clearly cannot be fixed under multiplication by any nonunit in $\mc{R}-\{0\}$. On the other hand, $\mc{A}[\mf{n}]$ is fixed by all elements of $\mc{R}-\{0\}$ prime to $\mf{n}$. It is a sum of irreducible cycles that freely generate $H^0(\mc{A}[\mf{n}],\Z)$. The sum of all $(\mc{R}-\{0\})$-multiples of such an irreducible cycle over elements prime to $\mf{n}$ provides a trace-fixed class, from which the result is clear (and
	did not require working with $\Z'$-coefficients).
\end{proof}

It follows from \cite[Proposition 6.1.2]{sv} that $H^i(\mc{A},\Z(2))^{(0)} = 0$ for all $i \neq 4$, and
$$
	H^4(\mc{A},\Z(2))^{(0)} \cong \Z'.
$$ 
In fact, this is already true for $(\Z-\{0\})$-fixed parts. The resulting surjection $\ms{K}_0(\mc{A})^{(0)} \to \Z'$ takes the class of a trace-fixed cycle to the order of its group of $\C$-points. We call this map the \emph{degree map}.

\begin{lemma} \label{exactK0}
	The image of the residue map 
	$\ms{K}_1(\mc{A})^{(0)} \to \ms{K}_0(\mc{A})^{(0)}$ 
	is the kernel of the degree map.
\end{lemma}

\begin{proof}
	Let $C$ be the span under prime-to-$\mf{n}$ multiplication maps of a connected component of $A_2[\mf{n}]$ for positive integer $\mf{n}$,
	 which we refer
	to as a component in this proof. (We eschew any analysis of these, as it is not
	required for our purposes.)
	As in \cite[1.10]{kato} and \cite[(6.5)]{sv}, if we consider the scheme
	$A_1 \times C$ (omitting the subscript $L$ on the product), then since the cycle 
	$C$ is fixed by prime-to-$\mf{n}$ multiplication,
	we have an exact sequence 
	\begin{equation} \label{ex_seq_rank_one}
		0 \to H^1((A_1 - A_1[\mf{n}]) \times C,1)^{(0)}
		\to H^0(A_1[\mf{n}] \times C,0)^{(0)}
		\to \Z',
	\end{equation}
	with the final map the degree map. This forms a subcomplex
	of $\ms{K}(\mc{A})^{(0)} \to \Z'$.

	The components of $\mc{A}[\mf{n}]$
	have the form $C \times D$, where $C$ is a component of $A_1[\mf{n}]$
	and the $D$ is a component of $A_2[\mf{n}]$. The degree of 
	$C \times D$ is the product of the degrees of $C$ and $D$. It suffices
	to see that every element of degree zero in $H^0(\mc{A}[\mf{n}],0)^{(0)}$ is a sum of elements of degree zero in 
	$H^0(C \times A_2[\mf{n}],0)^{(0)}$ and $H^0(A_1[\mf{n}] \times D,0)^{(0)}$ for some $C$ and $D$.
    For this, note that the classes of $\{0\}$ in $A_1[\mf{n}]$ and $A_2[\mf{n}]$ have degree one.

    Suppose we give each row and column of a matrix of a certain size a fixed positive integral weight, 
	with the first row and column having weight $1$, and we define the weight of an entry as the product of these.
	The result then amounts to the fact any such integral matrix with weighted sum of its entries equal to zero is a 
	sum of two matrices, one in which each row has weighted sum zero and one in which each column does. 
	In fact, one can choose the latter matrix to be zero outside of the first column.
\end{proof}

We also have the following.

\begin{proposition} \label{Kexact}
	The sequence
	$$
		0 \to \ms{K}_2(\mc{A})^{(0)} \to \ms{K}_1(\mc{A})^{(0)} \to \ms{K}_0(\mc{A})^{(0)} \to \Z' \to 0
	$$
	is exact.
\end{proposition}	

\begin{proof}
	Left exactness is \cite[Lemma 6.2.1]{sv}, surjectivity of the degree map holds as the class of $0 \in \mc{A}$ 
	has degree one, and exactness at $\ms{K}_0(\mc{A})^{(0)}$ is Lemma \ref{exactK0}.
\end{proof}

\subsection{Eisenstein cocycles for products of CM curves}

Let us return to our situation of interest, using the notation of Section \ref{CM_EC}. For $i = (i_1,i_2) \in I^2$, set
$$
	\mc{E}_i = E_{i_1} \times_L E_{i_2}.
$$
Elements $(a_{u,v})_{u,v}$ of the monoid $\Delta_{i,j}$ give rise to morphisms $\mc{E}_i \to \mc{E}_j$ of abelian $L$-schemes for $i, j \in I^2$ via the maps $a_{u,v} \colon E_{i_u} \to E_{j_v}$.
We set $\mc{E} = \coprod_{i \in I^2} \mc{E}_i$ for convenience.

Let us set $\ms{K}(i) = \ms{K}(\mc{E}_i)$. Pullback by elements of $\Delta_{i,j}$ provides morphisms of complexes $\ms{K}(j)$ to $\ms{K}(i)$, compatible with composition, which is to say that the $\ms{K}(i)$ form a complex of $\Delta$-module systems. 
The complexes $\ms{K}(i)^{(0)}$ are still of $\Delta$-module systems, as the diagram giving the two compositions $\mc{E}_i \to \mc{E}_j$ of $g \in \Delta_{i,j}$ with multiplication by an $\alpha \in \mc{O}$ that is prime to the ideal attached to $\det (x_igx_j^{-1})$ is cartesian, as in the proof of \cite[Lemma 6.3.1]{sv}.

With the notation of Section \ref{modsys}, we take $U = \GL_2(R)$ and $\tilde{\Delta} = M_2(R) \cap \GL_2(\mb{A}_F^f)$, so $\Gamma_i$ consists
of the elements of $\Delta_i$ with determinant in $\mc{O}^{\times}$. As a consequence of Proposition \ref{Kexact}, we have the connecting map
$$
	\ms{d}(i) \colon \ker(\ms{K}_0(i)^{(0)} \to \Z')^{\Gamma_i} \to 
	H^1(\Gamma_i,\ms{K}_2(i)^{(0)})
$$
in $\Gamma_i$-cohomology.
Since the differentials in the complex $\ms{K}(i)$ are $\Delta_{i,j}$-compatible, $\ms{d}(i)$ and $\ms{d}(j)$ are compatible with the Hecke operator $T(g)$, as noted in Section \ref{modsys}.

Let $(0) \in H^0(\mc{E},0)^{(0)}$ denote the sum of the classes of $0 \in \mc{E}_i$ over $i \in I^2$.
For a nonzero ideal $\mf{c}$ of $\mc{O}$, we set 
$$
	{}_{\mf{c}} e = ({}_{\mf{c}} e_i)_{i \in I^2} = 
	N\mf{c}^2(0)-\mc{E}[\mf{c}]  \in \bigoplus_{i\in I^2} \ker(\ms{K}_0(i)^{(0)} \to \Z')^{\Gamma_i}.
$$
Then the direct sum $\ms{d}$ of the $\ms{d}(i)$ applied to ${}_{\mf{c}} e$ is a class
$$
	{}_{\mf{c}} \Theta = ({}_{\mf{c}} \Theta_i)_{i \in I^2} \in \bigoplus_{i \in I^2} H^1(\Gamma_i,\ms{K}_2(i)^{(0)}).
$$ 
If $\mf{c} = \mc{O}$, then this class is zero.

\begin{lemma} \label{push_fixed}
    Let $i,j \in I^2$ and $\mf{n}$ be an ideal of $\mc{O}$ 
    prime to $\mf{c}$ 
    such that $i \sim_{\mf{n}} j$.
    Then
    \begin{equation} 
        [\mf{n}]_* ({}_{\mf{c}} \Theta_i) = {}_{\mf{c}} \Theta_j
    \end{equation}
    in $H^1(\Gamma_j, \ms{K}_2(j)^{(0)})$
\end{lemma}

\begin{proof}
    Since $\mf{n}$ is prime to $\mf{c}$,  we have $[\mf{n}]_*({}_{\mf{c}} e_i) = {}_{\mf{c}} e_j$. The base change condition in \eqref{gamma_commute} holds 
    on motivic cohomology as $\rho$ is proper and $\gamma$ is flat (in fact, an isomorphism), and the commutative square given by the two compositions 
    $\gamma \circ \rho = \rho \circ (\rho^{-1}\gamma\rho)$ is cartesian.
    The lemma then follows by Lemma \ref{pushforward_ideal}.
\end{proof}

For any nonzero ideal $\mf{n}$ of $\mc{O}$, we can view $T_{\mf{n}}$ and $[\mf{n}]^*$ as acting on  $\bigoplus_{i \in I^2} H^q(\Gamma_i,\ms{K}_p(i)^{(0)})$ for any $q \ge 0$ and $0 \le p \le 2$, and these operators act compatibly with
residues.

\begin{lemma} \label{Hecketors}
    Let $\mf{c}$ be a nonzero ideal of $\mc{O}$, and let
	$\mf{p}$ be a prime ideal of $\mc{O}$. Then $[\mf{p}]^* \mc{E}[\mf{c}] = \mc{E}[\mf{pc}]$, and $T_{\mf{p}} - (N\mf{p} + [\mf{p}]^*)$
	annihilates $\mc{E}[\mf{c}]$.
\end{lemma}

\begin{proof}
	Let $i,j \in I^2$ be such that $\mf{p}\mf{a}_{i_u,j_u}$ is principal, and let $\eta_u \in F$ be a generator,
    for $u \in \{1,2\}$.
	We can then view $\eta_u$ as an element of $\Hom_L(E_{i_u},E_{j_u})$ under its identification with $\mf{a}_{i_u,j_u}$
	from \eqref{isogeny_gp}. The pullback of $\mc{E}_j[\mf{c}]$ by $\smatrix{\eta_1 \\ & \eta_2}$ is then $\mc{E}_i[\mf{pc}]$, so
	$[\mf{p}]^*\mc{E}[\mf{c}] = \mc{E}[\mf{pc}]$. 

    Now let $i,j \in I^2$ with $i_2=j_2$ and $\mf{p}\mf{a}_{i_1,j_1} = (\eta_1)$.
    Since $T_{\mf{p}} = T(g)$ for $g = \smatrix{\eta_1 \\ & 1}$, and $g$ commutes with the diagonal matrix defining
    $[\mf{c}]^*$, it suffices to consider the case $\mf{c} = \mc{O}$.
    The right action of $\Gamma_i$ on the pullback $E_{i_1}[\mf{p}] \times \{0\}$ 
    of $0 \in \mc{E}_i$ by $g$
    factors through the quotient of $\Gamma_i$ by $\Gamma_i \cap (1+\mf{p}\Delta_i)$. This gives a compatible action of $U_i$ through the
    isomorphic quotient of $x_i^{-1}\GL_2(R_{\mf{p}})x_i$.
    Under the resulting pullback action, the matrices 
    $x_i^{-1} \smatrix{1 & a \\ 0 & 1} x_i$ with 
    $a$ running through representatives of $\mc{O}/\mf{p}$, together with $x_i^{-1} \smatrix{0 & 1 \\ 1 & 0} x_i$, carry $E_{i_1}[\mf{p}] \times \{0\}$ to the $N\mf{p}+1$ distinct $\mc{O}$-submodule schemes of $\mc{E}_i[\mf{p}]$ isomorphic to $\mc{O}/\mf{p}$. Since $T_{\mf{p}} = T_{\mb{A}}(\smatrix{\pi & 0 \\ 0 & 1})$
    for $\pi$ a uniformizer of $R_{\mf{p}}$ and
    $\smatrix{\pi&a\\0&1} = \smatrix{1&a\\0&1}\smatrix{\pi&0\\0&1}$
    for $a$ as above, while $\smatrix{1&0\\0&\pi} = \smatrix{0&1\\1&0}\smatrix{\pi&0\\0&1}\smatrix{0&1\\1&0}$, the sum of the
    classes of these subschemes gives exactly
    the result of the Hecke action on the class of $0$ in $\mc{E}_j$.
    We conclude that  $T_{\mf{p}}(0) = N\mf{p}(0) + \mc{E}[\mf{p}]$, 
	as desired.
\end{proof}

\begin{proposition} \label{big_cocycle_Eisenstein}
	For any prime ideal $\mf{p}$ and nonzero 
	ideal $\mf{c}$ of $\mc{O}$, we have $T_{\mf{p}}({}_{\mf{c}} \Theta) = (N\mf{p} + [\mf{p}]^*){}_{\mf{c}} \Theta$.
\end{proposition}

\begin{proof}
	By Lemma \ref{Hecketors}, we have 
	$T_{\mf{p}}({}_{\mf{c}} e) = (N\mf{p} + [\mf{p}]^*){}_{\mf{c}} e$.
	The result then follows from the Hecke-equivariance of the sum $\ms{d}$ of differentials.
\end{proof}

We can also describe the action of diamond operators.

\begin{lemma} \label{diamond_triv}
	Let $\mf{d}$ be an ideal of $\mc{O}$ prime to $\mc{N}$. Then the diamond operator $\langle \mf{d} \rangle^*$ 
	fixes $(0) \in H^0(\mc{E},0)$.
\end{lemma}

\begin{proof}
	Recall that $\delta \in \Delta_{i,j}$ with $\langle \mf{d} \rangle^* = T(\delta)$ has determinant generating $\mf{a}_{i_1,j_1}\mf{a}_{i_2,j_2}$,
	and $\delta^{-1} \in \Delta_{j,i}$ by Lemma \ref{diamond_lem}(b). 
	In particular, if $(x,y) \in \C \times \C$ is such that $(x,y)\delta \in \mf{a}_{j_1} \times \mf{a}_{j_2}$,
	then $(x,y) \in \mf{a}_{i_1} \times \mf{a}_{i_2}$. Therefore, pullback by $\delta$ on $H^0(\mc{E}_j,0)$ takes the
	class of $0$ to the class of $0$ in $H^0(\mc{E}_i,0)$.
\end{proof}

\begin{corollary} \label{cocyc_fixed_by_diamond}
	Let $\mf{d}$ be an ideal of $\mc{O}$ prime to $\mc{N}$, and let
	$\mf{c}$ be a nonzero ideal of $\mc{O}$. 
	Let $i, j \in I^2$ be such that $\mf{a}_{i_1,j_1}\mf{d}^{-1}$ and $\mf{a}_{i_2,j_2}\mf{d}$ are principal.
	Then $\langle \mf{d} \rangle^* {}_{\mf{c}} \Theta_j = {}_{\mf{c}} \Theta_i$.
\end{corollary}

\begin{lemma} \label{compare_classes}
	Let $i, j \in I^2$ and $\mf{n}$ be a nonzero ideal of $\mc{O}$ such that
	$i \sim_{\mf{n}} j$. Let $\mf{c}$ be a nonzero ideal of $\mc{O}$. Then we have
	$$
		[\mf{n}]^* ({}_{\mf{c}} \Theta_j) = {}_{\mf{cn}} \Theta_i - N\mf{c}^2 \cdot {}_{\mf{n}} \Theta_i
	$$
	in $H^1(\Gamma_i,\ms{K}_2(i))$.
\end{lemma}

\begin{proof}
	This is immediate from the fact that $[\mf{n}]^* ({}_{\mf{c}} e_j) = {}_{\mf{cn}} e_i - N\mf{c}^2 {}_{\mf{n}} e_i$.
\end{proof}

\subsection{A representative cocycle}

For $i \in I^2$, let us define $\ms{K}_{2,\mc{N}}(i)$ as the trace-fixed part of the
direct limit of groups $H^2(V,2)$ running over open subschemes $V$ of $\mc{E}_i$ which contain all $(\mc{O}/\mc{N})^{\times}$-multiples of $(0,P_{i_2})$. As $\Delta_0(\mc{N})_i$ preserves the latter set, $(\ms{K}_{2,\mc{N}}(i))_{i \in I^2}$ is a $\Delta_0(\mc{N})$-module system, and so is the collection of trace-fixed parts. To pull back by $(0,P_{i_2})$, we need a representative of the class ${}_{\mf{c}} \Theta_i$ that takes values in $\ms{K}_{2,\mc{N}}(i)$.

Let $\mf{c}$ be a ideal of $\mc{O}$ prime to $\mc{N}$.
Recall that for $r \in I$, there exists a unique trace-fixed element ${}_{\mf{c}} \theta_r \in H^1(E_r-E_r[\mf{c}],\Z'(1))$ with boundary
$N\mf{c} (0) - E_r[\mf{c}]$ (again, see \cite[1.10]{kato} and \cite[(6.5)]{sv}). For $i \in I^2$, the element
$$
	 {}_{\mf{c}} \vartheta_i = ({}_{\mf{c}} \theta_{i_1} \boxtimes E_{i_2}[\mf{c}]) + (N\mf{c} (0) \boxtimes {}_{\mf{c}} \theta_{i_2}),
$$
where $\boxtimes$ denotes the exterior product on motivic cohomology,
has boundary $e_{\mf{c}}$. However, 
its second term is problematic, as it is a unit on $\{0\} \times (E_{i_2} - E_{i_2}[\mf{c}])$, which contains $(0,P_{i_2})$.
We can avoid this issue by a minor adjustment.  
We instead consider
\begin{equation} \label{vartheta}
	\mu_i^* {}_{\mf{c}} \vartheta_i = ({}_{\mf{c}} \theta_{i_1} \boxtimes E_{i_2}[\mf{c}]) + \mu_i^*(N\mf{c} (0) \boxtimes {}_{\mf{c}} \theta_{i_2}),
\end{equation}
for a choice of matrix $\mu_i = \smatrix{ 1& 0 \\ x & 1 } $ with $x \in \mf{a}_{i_2,i_1}\mf{c}$ prime to $\mc{N}$, as the second term is then supported on $(\{0\} \times (E_{i_2} - E_{i_2}[\mf{c}]))\mu_i^{-1}$, while $\mu_i^*$ leaves its residue unchanged.

We may then view ${}_{\mf{c}} \Theta_i$ as the class of the cocycle
$$
	{}_{\mf{c}} \Theta_i \colon \Gamma_i \to \ms{K}_2(i)^{(0)}
$$
uniquely determined by 
$$
	\partial({}_{\mf{c}} \Theta_i(\gamma)) = (\gamma^*-1)\mu_i^*({}_{\mf{c}} \vartheta_i)
$$
for $\gamma \in \Gamma_i$. 
Whether ${}_{\mf{c}} \Theta$ is being used to denote a collection of cocycles or their classes
should be gleaned from context: for instance, when studying the action of Hecke operators, we are referring to the class,
whereas when speaking of the values of ${}_{\mf{c}} \Theta_i$, we are referring to the cocycle.

Moreover, we have the following.

\begin{lemma} \label{imageGamma0}
    The restriction of ${}_{\mf{c}} \Theta_i$ to $\Gamma_0(\mc{N})_i$ takes values in
    $\ms{K}_{2,\mc{N}}(i)$.
\end{lemma}

\begin{proof}
    For $\gamma = \smatrix{a&b\\c&d} \in \Gamma_i$, set $S_{\gamma} = 
    (E_{i_1} \times E_{i_2}[\mf{c}]) \gamma^{-1}$ and $S'_{\gamma} = (\{0\} \times E_{i_2})\gamma^{-1}$. 
    Our choice of cocycle ${}_{\mf{c}} \Theta_i$ is determined uniquely
    by the fact that its values are trace fixed and 
    the residue of its value on $\gamma \in \Gamma_i$
    is $(\gamma^*-1)\mu_i^*{}_{\mf{c}} \vartheta_i$. This value
    lies in $H^2(V_{\gamma},\Z(2))^{(0)}$ for $V_{\gamma}$ the complement in $\mc{E}_i$ of the four codimension one subvarieties
    $S_1$, $S_{\gamma}$, $S'_{\mu_i}$, and $S'_{\gamma\mu_i}$. For $\gamma = \smatrix{a&b\\c&d} \in 
    \Gamma_0(\mc{N})_i$, each of $1$, $d$, $x$ and $c + dx$ 
    is nonzero modulo $\mc{N}$, so this cohomology group
    is a subgroup of $\ms{K}_{2,\mc{N}}(i)^{(0)}$.
\end{proof} 

\begin{remark} \label{Eisenstein_restrict}
	The group $\ms{K}_{2,\mc{N}}(i)$ fits in a $\Delta_0(\mc{N})_i$-equivariant subcomplex $\ms{K}_{\mc{N}}(i)$ 
	of $\ms{K}(i)$ which in degrees $p \in \{1,0\}$ consists of the trace-fixed part of the
    sums of $K$-groups of the irreducible
	$p$-cycles not intersecting the $(\mc{O}/\mc{N})^{\times}$-orbit of $(0,P_i)$ in $\mc{E}_i$. 
	Together for all $i$, these give a $\Delta_0(\mc{N})$-module system. Since ${}_{\mf{c}} e_i \in \ms{K}_{0,\mc{N}}(i)$ and 
	$\mu_i^*{}_{\mf{c}} \vartheta_i \in \ms{K}_{1,\mc{N}}(i)$, using the complex $\ms{K}_{\mc{N}}(i)$,  we get that 
	the class ${}_{\mf{c}} \Theta_i|_{\Gamma_0(\mc{N})_i}$ and its explicit 
	representative have canonical lifts valued in $\ms{K}_{2,\mc{N}}(i)$. 
	In particular, since the double coset decompositions of $T_{\mf{p}}$
	are unchanged upon passage from the groups $\Gamma_i$ to their subgroups $\Gamma_0(\mc{N})_i$ for 
	primes $\mf{p} \nmid \mc{N}$, we still have the Eisenstein property of ${}_{\mf{c}} \Theta_i|_{\Gamma_0(\mc{N})_i}$ 
	viewed as a class in $H^1(\Gamma_0(\mc{N})_i,\ms{K}_{2,\mc{N}}(i))$ for such operators.
\end{remark}

The above construction of ${}_{\mf{c}}\Theta_i$ can be carried out for any product $\mc{A} = A_1 \times A_2$ of elliptic curves with CM by $\mc{O}$, defined over some extension of the Hilbert class field of $F$. As such, the following lemma holds.

\begin{lemma} \label{cocycle_upon_isom}
	Let $A_1$ and $A_2$ be elliptic curves with CM by $\mc{O}$ defined over
    a field $L$ containing $F(\mf{f})$, and suppose that there exist $i \in I^2$ 
	and $L$-isomorphisms $\psi_u \colon A_u \xrightarrow{\sim} E_{i_u}$ for $u \in \{1,2\}$.  Let $\psi = (\psi_1, \psi_2)$, which
	induces an action of $\Gamma_i$ on $\mc{A} = A_1 \times A_2$. Then the cocycle ${}_{\mf{c}} \Theta_{\mc{A}} \colon \Gamma_i \to K_2(L(\mc{A}))^{(0)}$ attached to $N\mf{c}^2(0)-\mc{A}[\mf{c}]$,
    with residue the analogue of \eqref{vartheta} for the same $\mu_i$,
    satisfies 
    ${}_{\mf{c}} \Theta_{\mc{A}} = \psi^* {}_{\mf{c}} \Theta_i$.
\end{lemma}

\section{Specialized cocycles for ray class fields}

\subsection{Pullback by $\mc{N}$-torsion}

Let us now choose $\mc{N}$ so that $\mc{O}^{\times} \to (\mc{O}/\mc{N})^{\times}$ is injective.
For $\alpha \in F^{\times}$ prime to $\mathcal{N}$ and $i \in I^2$, let $[\alpha]_i \colon \mc{E}_i[\mathcal{N}] \rightarrow \mc{E}_i[\mathcal{N}]$ be multiplication by any element in $\mathcal{O}$ congruent to $\alpha$ modulo $\mathcal{N}$.
Let us view $(0,P_{i_2})$ as a morphism 
$\lambda_i \colon \Spec F(\mc{N}\cap \mf{f}) \to \mc{E}_i[\mc{N}]$. 
It then defines a pullback map
$$
    \lambda_i^* \colon \ms{K}_{2,\mc{N}}(i)
    \to K_2(F(\mc{N} \cap \mf{f}))_{\Z'}.
$$

\begin{lemma} \label{Galois_fixed_image}
	The image of $\lambda_i^* \circ {}_{\mf{c}} \Theta_i|_{\Gamma_0(\mc{N})_i}$ is $\Gal(F(\mc{N} \cap \mf{f})/F(\mc{N}))$-fixed.
\end{lemma}

\begin{proof}
	For $j \in \{1,2\}$, choose isomorphisms 
    $\psi_j \colon A_j \to E_{i_j}$ 
	with elliptic curves $A_j$ with CM by $\mc{O}$ that are 
    defined over $F(\mc{N})$ and have torsion contained in $F^{\ab}$.
	These are necessarily $F(\mc{N} \cap \mf{f})$-isomorphisms 
    (noting Lemma \ref{isog_field_of_def})
    as the compositions of the Hecke characters of $A_j$ and $E_{i_j}$ with norms from $F(\mc{N} \cap \mf{f})$ agree by part (i) of the lemma of \cite[1.4]{deshalit}. Note in particular that the $\mc{N}$-torsion in each $A_j$ is defined over $F(\mc{N})$.
	If $P = \psi_2^{-1}(P_{i_2})$, then by and in the notation of Lemma \ref{cocycle_upon_isom}, we have
	$$
		(0,P)^* \circ {}_{\mf{c}} \Theta_{\mc{A}} = (0,P)^* \circ\psi^* \circ {}_{\mf{c}} \Theta_i = \lambda_i^* \circ {}_{\mf{c}} \Theta_i
	$$
	in the sense that $\lambda_i^* \circ {}_{\mf{c}} \Theta_i \colon \Gamma_0(\mc{N})_i \to K_2(F(\mc{N} \cap \mf{f}))_{\Z'}$ takes
	values in the image of the codomain of $(0,P)^* \circ {}_{\mf{c}} \Theta_{\mc{A}}$, which is $K_2(F(\mc{N}))_{\Z'}$.
\end{proof}

We now take $\mf{f}$ to be relatively prime to $\mc{N}$.

\begin{proposition} \label{specialize_matrix_act}
	For $i,j \in I^2$ and $\smatrix{ a&b\\c&d } \in \Delta_0(\mc{N})_{i,j}$, we have the equality
	$$
		\lambda_i^* \circ \smatrix{a&b\\c&d}^*= \mc{R}(d') \circ \lambda_j^*
	$$
	of specialization maps on $\ms{K}_{2,\mc{N}}(j)$, where
	$d' \in \mc{O}$ is such that $d' \equiv d \bmod \mc{N}$ and $d' \equiv 1 \bmod \mf{f}$.
\end{proposition}

\begin{proof}
	First, note that $\lambda_i^* \circ \smatrix{a&b\\c&d}^* = (\smatrix{a&b\\c&d} \circ \lambda_i)^*$ 
	agrees with pullback by $(0, d \cdot P_{i_2})$.
	Corollary \ref{action_on_pts} and Lemma \ref{action_on_pts_single_curve} tell us that 
	$$
		d \cdot P_{i_2} = [d]_{j_2} P_{j_2} = [d']_{j_2}P_{j_2} = \mc{R}(d')(P_{j_2}).
	$$
	We therefore have the first equality in
	$$
		\lambda_i^* \circ \smatrix{a&b\\c&d}^* = (0,\mc{R}(d')(P_{j_2}))^* = \mc{R}(d') \circ \lambda_j^*,
	$$
	where for the second equality, we have used the fact that 
	$\ms{K}_{2,\mc{N}}(i)$ is generated by classes of cycles defined over $F(\mf{f})$, 
    which are fixed under the action of $\mc{R}(d')$.
\end{proof}

Let $\mc{R}_{\mc{N}}$ denote the Artin map from the $\mc{N}$-ray class group $\Cl_{\mc{N}}(F)$ of $F$ to $\Gal(F(\mc{N})/F)$.
Lemma \ref{Galois_fixed_image} and Proposition \ref{specialize_matrix_act} combine to provide the following.

\begin{corollary} \label{act_special_cocycle}
	For $i,j \in I^2$ and $\smatrix{ a&b\\c&d } \in \Delta_0(\mc{N})_{i,j}$, we have
	$$
		\lambda_i^* \circ \smatrix{a&b\\c&d}^* \circ {}_{\mf{c}} \Theta_j = \mc{R}_{\mc{N}}(d) \circ \lambda_j^*  \circ 
		{}_{\mf{c}} \Theta_j
	$$
	on $\Gamma_0(\mc{N})_j$.
\end{corollary}

We then have a specialized cocycle that is independent of $\mf{f}$.

\begin{proposition} \label{indep_of_f}
	There exists a cocycle
	$$
		{}_{\mf{c}} \Theta_{i,\mc{N}} \colon \Gamma_0(\mc{N})_i \to K_2(F(\mc{N}))_{\Z'}.
	$$
	such that for every choice of $\mf{f}$, the cocycles ${}_{\mf{c}} \Theta_{i,\mc{N}}$ and $\lambda_i^* \circ {}_{\mf{c}} \Theta_i$
	agree as maps to $K_2(F(\mc{N}\mf{f})) \otimes \Z'[\tfrac{1}{f}]$, where $f$ is the order of $(\mc{O}/\mf{f})^{\times}$.
\end{proposition}

\begin{proof}
	Since $[F(\mc{N}\mf{f}):F(\mc{N})]$ divides $f$, we have
	$$
		K_2(F(\mc{N}\mf{f}))^{\Gal(F(\mc{N}\mf{f})/F(\mc{N}))} \otimes \Z'[\tfrac{1}{f}] \cong
		K_2(F(\mc{N})) \otimes \Z'[\tfrac{1}{f}].
	$$
	By Lemma \ref{Galois_fixed_image}, we may therefore speak of the pullback 
	$\lambda_i^* \circ {}_{\mf{c}} \Theta_i$ as taking values in $K_2(F(\mc{N})) \otimes \Z'[\tfrac{1}{f}]$
	for $i \in I^2$. It is then a cocycle for the action of $\smatrix{a&b\\c&d} \in \Gamma_i$ on $K_2(F(\mc{N}))$ 
	by $\mc{R}_{\mc{N}}(d)$ by Corollary \ref{act_special_cocycle}.
	By construction and Lemma \ref{cocycle_upon_isom}, this cocycle is, up to the inversion of
	$f$, independent of the choice of $\mf{f}$. Varying $\mf{f}$ (or even just taking it to be a sufficiently large
	power of a prime over $2$ prime to $\mc{N}$ if such a prime exists), we obtain the claimed well-defined cocycle.
\end{proof}

The following corollary is then immediate from the definitions.

\begin{corollary} \label{Hecke_specialization}
	For $i,j \in I^2$ and $g \in \Delta_0(\mc{N})_{i,j}$, we have
	$$ 
		T(g)({}_{\mf{c}} \Theta_{j,\mc{N}}) = \lambda_i^* \circ T(g)({}_{\mf{c}}\Theta_j|_{\Gamma_0(\mc{N})_j})
	$$
	in $H^1(\Gamma_0(\mc{N})_i, K_2(F(\mc{N}))_{\Z'})$.
\end{corollary}

Proposition \ref{big_cocycle_Eisenstein}, noting Remark \ref{Eisenstein_restrict}, translates to give the Eisenstein property of the specialized cocycles. For this, note that if $\mf{p}$ is prime to $\mc{N}$ and $\mf{a}_{i_1,j_1}\mf{p} = (\eta)$ and $i_2 = j_2$, then the double cosets $\Gamma_i \smatrix{\eta & 0 \\ 0 & 1} \Gamma_j$ and $\Gamma_0(\mc{N})_i\smatrix{\eta & 0\\ 0 & 1} \Gamma_0(\mc{N})_j$ have sets of left coset representatives that are equal.

\begin{corollary} \label{spec_cocyc_Eis}
	For any prime ideal $\mf{p}$ of $\mc{O}$ not dividing $\mc{N}$, the operator $T_{\mf{p}}-(N\mf{p}+[\mf{p}]^*)$
	annihilates ${}_{\mf{c}} \Theta_{\mc{N}} = ({}_{\mf{c}} \Theta_{i,\mc{N}})_{i \in I^2}$, 
	and every diamond operator $\langle \mf{d} \rangle^*$ acts trivially on ${}_{\mf{c}} \Theta_{\mc{N}}$.
\end{corollary}

The following lemma will be useful for us.

\begin{lemma}\label{key_lemma_Galois_push}
    Let $i,j \in I^2$ and an ideal $\mf{n}$ of $\mc{O}$ prime to $\mc{N}\mf{c}$ 
    be such that $i \sim_{\mf{n}} j$. For $u \in \{1,2\}$, let $\eta_j \in F^{\times}$
    be a generator of $\mf{a}_{i_u,j_u}\mf{n}$.
    Then the identity
    $$
		\lambda_j^* \circ \smatrix{\eta_1 & 0 \\ 0 & \eta_2}_* 
        \circ {}_{\mf{c}} \Theta_i = \mc{R}_{\mc{N}}(\mf{a}_{i_2,j_2})^{-1} 
        \circ \lambda_i^* \circ {}_{\mf{c}} \Theta_i
	$$
    is satisfied on $\Gamma_0(\mc{N})_i$.
\end{lemma}

\begin{proof}
For $u \in \{1,2\}$, the theory of complex multiplication as in \cite[1.5]{deshalit} provides unique $F(\mf{f})$-isogenies $E_{i_u} \to E_{i_u}^{\mc{R}(\mf{n})}$ with kernel $E_{i_u}[\mf{n}]$ that agree with $\mc{R}(\mf{n})$ on prime-to-$\mf{n}$ torsion. Let $\phi_i^{(\mf{n})} \colon \mc{E}_i \to \mc{E}_i^{\mc{R}(\mf{n})}$ denote the resulting isogeny. We then have an isomorphism $\psi \colon \mc{E}_i^{\mc{R}(\mf{n})}
\to \mc{E}_j$ such that $\psi \circ \phi_i^{(\mf{n})} = \rho \colon \mc{E}_i \to \mc{E}_j$.
Setting $\lambda_i^{(\mf{n})} = \phi_i^{(\mf{n})} \circ \lambda_i$, Corollary \ref{action_on_pts} and Lemma \ref{action_on_pts_single_curve} imply that $\psi \circ \lambda_i^{(\mf{n})} = \mc{R}(\eta_2') \circ\lambda_j$ for any $\eta_2'\in \mc{O}$ with $\eta_2' \equiv 1 \bmod \mf{f}$ and $\eta_2' \equiv \eta_2 \bmod \mc{N}$.
We then have
\begin{equation}\label{eq_rho_phi}
        (\mc{R}(\eta_2')\lambda_j)^* \circ \rho_* 
        = (\mc{R}(\eta_2')\lambda_j)^* \circ \psi_*
        \circ (\phi_i^{(\mf{n})})_* 
        = (\lambda_i^{(\mf{n})})^*\circ (\phi_i^{(\mf{n})})_*  
\end{equation}
on $\ms{K}_{2,\mc{N}}(i)$.

We have a canonical identification of $\Aut(\mc{E}_i^{\mc{R}(\mf{n})})$ with $\Gamma_i$ given by applying the Galois element $\mc{R}(\mf{n})$ to the automorphism of $\mc{E}_i$ defined by an element of $\Gamma_i$. This gives the motivic complex $\ms{K}'(i)
= \mc{K}(\mc{E}_i^{\mc{R}(\mf{n})})$ a left pullback action of $\Gamma_i$. Pushforward by the isogeny $\phi_i^{(\mf{n})}$ induces a morphism of complexes of $\Gamma_i$-modules between $\ms{K}(i)$ and $\ms{K}'(i)$, as does application of the
Galois element $\mc{R}(\mf{n})$ by the choice of the $\Gamma_i$-action on $\mc{E}_i^{\mc{R}(\mf{n})}$ we have taken. We let 
$${}_{\mf{c}} \Theta'_i = (\phi_i^{(\mf{n})})_* {}_{\mf{c}} \Theta_i \colon \Gamma_i \rightarrow \ms{K}'_2(i) ,$$ 
which is the the unique cocycle satisfying
$$\partial({}_{\mf{c}} \Theta_i'(\gamma)) = (\gamma^*-1)(\phi_i^{(\mf{n})})_*\mu_i^*({}_{\mf{c}} \vartheta_i) .$$
We claim that we have the following equality of cocycles on $\Gamma_i$:
\begin{equation}\label{eq_isogeny_Galois}
{}_{\mf{c}} \Theta_i' = \mc{R}(\mf{n}) \circ {}_{\mf{c}} \Theta_i .
\end{equation}
This reduces quickly to the equality $(\phi_i^{(\mf{n})})_*{}_{\mf{c}} \vartheta_i = \mc{R}(\mf{n}){}_{\mf{c}} \vartheta_i$
in $\ms{K}'_1(i)$. Note that ${}_{\mf{c}} \vartheta_i$ is a sum
of elements of
$H^1(E_{i_1}[\mf{c}] \times (E_{i_2}-E_{i_2}[\mf{c}]),1)^{(0)}$
and $H^1((E_{i_1}-E_{i_1}[\mf{c}]) \times E_{i_2}[\mf{c}],1)^{(0)}$,
and the residue map from each of these groups to $H^0(\mc{E}_i[\mf{c}],0)^{(0)}$ is injective. Thus, we have the
equality from the agreement of $\phi_i^{(\mf{n})}$ and $\mc{R}(\mf{n})$ on prime-to-$\mf{n}$ torsion.

Restricting the equality \eqref{eq_isogeny_Galois} to $\Gamma_0(\mc{N})_i$ 
and noting that $\mc{R}(\mf{n})(\lambda_i)
= \lambda_i^{(\mf{n})}$, we see noting Lemma \ref{Galois_fixed_image} that
\begin{equation} \label{compare_special}
(\lambda_i^{(\mf{n})})^* \circ {}_{\mf{c}} \Theta_i' = \mc{R}_{\mc{N}}(\mf{n}) \circ \lambda_i^* \circ {}_{\mf{c}} \Theta_i.
\end{equation}
 As $\mc{R}_{\mc{N}}(\eta_2) = \mc{R}_{\mc{N}}(\eta'_2)$ and $\mc{R}_{\mc{N}}(\mf{n})\mc{R}_{\mc{N}}(\eta_2)^{-1} = \mc{R}_{\mc{N}}(\mf{a}_{i_2,j_2})^{-1}$, combining \eqref{compare_special} with \eqref{eq_rho_phi} yields the desired identity.
\end{proof}

The following provides a direct connection between the classes ${}_{\mf{c}} \Theta_{i,\mc{N}}$ for equivalent $i \in I^2$.

\begin{proposition} \label{prop_pull_galois}
	For $i, j \in I^2$ and $\mf{n}$ an ideal of $\mc{O}$ prime to $\mc{N}$ 
	such that $i \sim_{\mf{n}} j$, 
	we have
	$$
		[\mf{n}]^* ({}_{\mf{c}} \Theta_{j,\mc{N}}) = \mc{R}_{\mc{N}}(\mf{n}) \circ {}_{\mf{c}} \Theta_{i,\mc{N}}
	$$
    in $H^1(\Gamma_0(\mc{N})_i,K_2(F(\mc{N}))_{\Z'})$.
\end{proposition}

\begin{proof}
We verify this in $H^1(\Gamma_0(\mc{N})_i,K_2(F(\mc{N}))) \otimes_{\Z} \Z'[\frac{1}{f}]$ for $f = |(\mc{O}/\mf{f})^{\times}|$ by working with the tuple $(\lambda_k^* \circ {}_{\mf{c}} \Theta_k)_{k \in I^2}$ for a given $\mf{f}$ prime to $\mc{N}$, as it agrees with $({}_{\mf{c}} \Theta_{k,\mc{N}})_{k \in I^2}$ upon inverting $f$. The result then follows by varying $\mf{f}$.

Choose an integral ideal $\mf{n}'$ coprime to $\mc{N}\mf{c}$ having the same ideal class as $\mf{n}$. Write $\mf{a}_{i_u,j_u}\mf{n}=(\eta_u)$ and $\mf{a}_{i_u,j_u}\mf{n}'=(\eta_u')$ for $u \in \{1,2\}$, chosen such that $\eta'_2\eta_1
= \eta_2\eta'_1$. Set $\rho = \smatrix{\eta_1 & 0 \\ 0 & \eta_2}$ and $\rho' = \smatrix{\eta_1' & 0 \\ 0 & \eta_2'}$.
Recall from Lemma \ref{push_fixed} that $[\mf{n}']_*({}_{\mf{c}} \Theta_i) = {}_{\mf{c}} \Theta_j$ as cohomology classes.
Since $\rho'\rho^{-1}$ is scalar, we also have
	\begin{align*}
		[\mf{n}]^*([\mf{n}']_*({}_{\mf{c}} \Theta_i))(\gamma) 
		&=  \rho^* \circ \rho'_* \circ {}_{\mf{c}}
		\Theta_i(\rho'\rho^{-1}\gamma\rho(\rho')^{-1}) \\
		&= \rho^* \circ \rho'_* \circ {}_{\mf{c}} \Theta_i(\gamma)
	\end{align*}
	for $\gamma \in \Gamma_i$. 
	Combining these identities with Corollary \ref{Hecke_specialization}, Proposition \ref{specialize_matrix_act} and Lemma 
 \ref{key_lemma_Galois_push}, we have
	\begin{align*}
		[\mf{n}]^*({}_{\mf{c}} \Theta_{j,\mc{N}}) &= \lambda_i^* \circ [\mf{n}]^*({}_{\mf{c}} \Theta_j) \\
		&= \lambda_i^* \circ [\mf{n}]^*([\mf{n}']_*({}_{\mf{c}} \Theta_i)) \\
		&= \lambda_i^* \circ \rho^* \circ \rho'_* \circ {}_{\mf{c}}\Theta_i \\
		&= \mc{R}(\mf{n}\mf{a}_{i_2,j_2}) \circ \lambda_j^* \circ \rho'_* \circ {}_{\mf{c}}\Theta_i \\
		&= \mc{R}_{\mc{N}}(\mf{n}) \circ {}_{\mf{c}} \Theta_{i,\mc{N}}.
	\end{align*}
\end{proof}

This allows us to describe the relationship between the classes ${}_{\mf{c}} \Theta_{i,\mc{N}}$ as we vary $\mf{c}$.

\begin{corollary} \label{compare_aux_ideal}
	Let $\mf{c}$ and $\mf{d}$ be ideals of $\mc{O}$ prime to $\mc{N}$. Then
	we have the equality of classes
	$$
		(N\mf{d}^2 - \mc{R}_{\mc{N}}(\mf{d})) {}_{\mf{c}} \Theta_{i,\mc{N}} 
		= (N\mf{c}^2 - \mc{R}_{\mc{N}}(\mf{c})) {}_{\mf{d}} \Theta_{i,\mc{N}}
	$$
	for each $i \in I^2$.
\end{corollary}

\begin{proof}
	Proposition \ref{prop_pull_galois} and Lemma \ref{compare_classes} combine 
    to tell us that 
	$$
		\mc{R}_{\mc{N}}(\mf{d}) \circ {}_{\mf{c}} \Theta_{i,\mc{N}}
		 = {}_{\mf{cd}} \Theta_{i,\mc{N}} - N\mf{c}^2 {}_{\mf{d}} \Theta_{i,\mc{N}},
	$$
	and similarly with $\mf{c}$ and $\mf{d}$ reversed, from which the identity follows.
\end{proof}

\subsection{Integrality}

Fix an ideal $\mf{c}$ of $\mc{O}$ prime to $\mc{N}\mf{f}$. Adopting the notation of the proof of Lemma \ref{imageGamma0}, we set
$T = S_1 \cup S'_{\mu_i}$. 
For $\gamma \in \Gamma_i$,
we set $T_{\gamma} = T \cup T \gamma^{-1}$ and
$V_{\gamma} = \mc{E}_i - T_{\gamma}$, and we let $T_{\gamma}^{\circ}$ be the complement in $T_{\gamma}$
of the pairwise intersections of $S_1$, $S_{\gamma}$, $S'_{\mu_i}$, and $S'_{\gamma\mu_i}$.

The elliptic curves $E_r$ for $r \in I$ have good reduction at all primes of the ring of integers $\mc{O}(\mf{f})$ of
$F(\mf{f})$ not dividing $\mf{f}$, as $\mf{f}$ is the conductor of the Hecke character of $E_r$. Let $\mathscr{e}_i$ denote the reduction of $\mc{E}_i$ modulo a prime $\mf{l}$ of $\mc{O}(\mf{f})$ not dividing $\mf{f}$.
Let $R$ be the localization of $\mc{O}(\mf{f})$ at $\mf{l}$, and let $\mc{E}_{i/R}$ be the product
of the Neron models of $E_{i_1}$ and $E_{i_2}$ over $R$. 
We use $v_{\gamma}$ and $t_{\gamma}^{\circ}$ (resp., $V_{\gamma/R}$ and $T_{\gamma/R}^{\circ}$) in place of $V_{\gamma}$ and $T_{\gamma}^{\circ}$, respectively, to denote the subschemes of $\mathscr{e}_i$ (resp., $\mc{E}_{i/R}$) constructed in the same manner as the so-denoted subschemes of $\mc{E}_i$.

We continue to assume that $\mf{f}$ is prime to $\mc{N}$.

\begin{lemma} \label{trace_comm_res}
    We have a commutative square of residue maps
    \begin{equation} \label{res_square}
	\begin{tikzcd}
		H^2(V_{\gamma},2) \arrow{r}{} \arrow{d}{} & H^1(v_{\gamma},1)\arrow{d}{} \\
		H^1(T_{\gamma}^{\circ},1) 
		\arrow{r}{} & H^0(t_{\gamma}^{\circ},0)
	\end{tikzcd}
    \end{equation}
    that commute with trace maps $[\alpha]_*$ for $\alpha \in \mc{O}$ prime to $\mf{l}$. 
\end{lemma}

\begin{proof}By \cite[Theorem 1.7]{levine-tech},
	We have a distinguished triangle of cycle complexes computing motivic cohomology:
	\begin{equation} \label{seq_complexes}
		0 \to z_q(v_{\gamma},*) \to z_q(V_{\gamma/R},*) \to z_q(V_{\gamma},*)
	\end{equation}
	and similarly for $T_{\gamma}^{\circ}$. This gives the horizontal
    residue maps in \eqref{res_square}.
    Since the maps in \eqref{seq_complexes} commute with those induced by multiplication by $\alpha \in \mc{O}$ prime to $\mf{l}$, we have
    the commutativity of the horizontal maps with $[\alpha]_*$.
    The vertical maps in \eqref{res_square} are the more usual residue maps
    over a common base field, which we already know commute with these
    trace maps.
\end{proof}

\begin{lemma} \label{triv_res}
	Let $\mf{l}$ be a prime of $\mc{O}(\mf{f})$ not dividing $\mf{f}$. 
	For each $\gamma \in \Gamma_i$, the value ${}_{\mf{c}} \Theta_i(\gamma)$ lies in the kernel of the residue map
	$H^2(V_{\gamma},\Z'(2)) \to H^1(v_{\gamma},\Z'(1))$.
\end{lemma}

\begin{proof}
	In the diagram \eqref{res_square}, 
	the two vertical maps
	are injective upon taking trace-fixed parts.
	To see this for the right-hand vertical map, the relevant coniveau spectral sequence (as employed in \cite[Section 2.2]{sv})
	tells us that the kernel is $H^1(\mathscr{e}_i,1)^{(0)}$, where $\mathscr{e}_i$ is the reduction of $\mc{E}_i$ modulo
	$\mf{l}$. Since 
	$\mathscr{e}_i$ is projective, 
	this group consists only of constant units, but as these are pulled back 
	from the residue field, $[\alpha]_*$ acts on them by raising to the power $N\alpha$. Thus, only the identity in $H^1(\mathscr{e}_i,1)^{(0)}$
    is trace-fixed.
	
	Recall that ${}_{\mf{c}} \Theta_i(\gamma)$ has residue
	 $(\gamma^*-1)\mu_i^* {}_{\mf{c}} \vartheta_i \in H^1(T_{\gamma}^{\circ},1)^{(0)}$
	for $\gamma \in \Gamma_i$. By Lemma \ref{trace_comm_res} and the injectivity of the righthand vertical map on trace-fixed parts,
	we need only show that each $\gamma^* ({}_{\mf{c}} \vartheta_i)$ for $\gamma \in \Gamma_i$ has trivial residue in
	$H^0(t_{\gamma}^{\circ},0)$. By the relevant Gysin sequence as in \cite[Corollary 3.4]{geisser}, it then suffices
	to check that each $\gamma^* ({}_{\mf{c}} \vartheta_i)$ for $\gamma \in \Gamma_i$ is the restriction of a unit 
	on $T_{\gamma/R}^{\circ}$.
	
	It is enough to see that ${}_{\mf{c}} \vartheta_i$ is a unit on $T^{\circ}_{/R}$, where $T^{\circ} = T_1^{\circ}$.	
	Since $\mf{l}$ does not divide the conductor $\mf{f}$ of the Hecke characters
	of $E_{i_1}$ and $E_{i_2}$, these elliptic curves have good reduction at $\mf{l}$.
	Then ${}_{\mf{c}} \vartheta_i$ extends over relevant open subscheme of the 
	fiber product of Neron models of $E_{i_1}$  and $E_{i_2}$ over 
	the localization of $\mc{O}(\mf{f})$ at $\mf{l}$, since the theta functions used in its definition \eqref{vartheta} do,
	as desired. 
\end{proof}

For any multiple $\mc{M}$ of $\mc{N}$, let us write
$$
	\mc{O}'(\mc{M}) = \begin{cases} \mc{O}(\mc{M})[\frac{1}{\mc{N}}] & \text{if $\mc{N}$ is a prime power}, \\
		\mc{O}(\mc{M}) & \text{otherwise}
		\end{cases}
$$
for brevity of notation in what follows.

\begin{lemma} \label{reduce_in_complement}
    Let $\mf{l}$ be a prime of $\mc{O}'(\mc{N}\mf{f})$ not lying over $\mf{f}$.
	For each $\gamma \in \Gamma_0(\mc{N})_i$ and primitive $\mc{N}$-torsion point $P$ of $E_{i_2}$, the $F(\mc{N}\mf{f})$-point $(0,P) \in V_{\gamma}$ reduces to a point of $v_{\gamma}$.
\end{lemma}

\begin{proof}
	It suffices to show that the mod $\mf{l}$ reduction of $(0,P)$ 
    does not lie in the reduction of $T \rho^{-1}$
	for $\rho \in \Gamma_0(\mc{N})_i$, as $T_{\rho} = T \cup T \rho^{-1}$.
	We can consider $\rho = 1$ by replacing $(0,P)$ by $(0,P_{i_2})\rho^{-1} = (0,P')$ for some primitive $\mc{N}$-torsion point $P' \in E_{i_2}$. It further suffices to show that the reduction of 
    $(0,P)$ does not lie in the reduction of $S_1 = E_{i_1} \times E_{i_2}[\mf{c}]$,
	and, translating by $\mu_{\mf{c}}^{-1} = \smatrix{1 & 0 \\ -x & 1}$, that the reduction of $(-x P, P)$ does not lie in the reduction of $S'_1
    = \{0\} \times E_{i_2}$. 
    The former amounts to showing that no point of $P + E_{i_2}[\mf{c}]$ reduces to zero at $\mf{l}$, which follows from \cite[Lemma 7.3(ii)]{rubin}. For the latter, we need only 
    note that $x$ is prime to $\mc{N}$ and $P$ is nonzero modulo $\mf{l}$
    by the same lemma.
\end{proof}

\begin{proposition} \label{integral_image}
	The cocycle ${}_{\mf{c}} \Theta_{i,\mc{N}}$ takes values in $K_2(\mc{O}'(\mc{N}))_{\Z'}$.
\end{proposition}

\begin{proof}
	Let $\mf{l}$ be a prime of $\mc{O}'(\mc{N}\mf{f})$ not dividing $\mf{f}$.
	By Lemma \ref{reduce_in_complement}, the reduction of $(0,P_{i_2})$ lies in $v_{\gamma}$ for each $\gamma \in \Gamma_0(\mc{N})_i$.
	We apply Lemma \ref{triv_res}, pull back by this reduced point and use the commutativity of pullbacks and residues to see that ${}_{\mf{c}} \Theta_{i,\mc{N}}(\gamma)$
	has trivial residue in the tensor product with $\Z'$ of 
	the first $K$-group of the residue field of $\mc{O}'(\mc{N}\mf{f})$ at $\mf{l}$.
    
	  We thus get that 
	${}_{\mf{c}} \Theta_{i,\mc{N}}(\gamma)$ lies in $K_2(\mc{O}'(\mc{N}\mf{f})[\tfrac{1}{\mf{f}}])_{\Z'}$, 
	but we also know it is fixed by $\Gal(F(\mc{N}\mf{f})/F(\mc{N}))$, so again by Galois descent it takes values in 
	$K_2(\mc{O}'(\mc{N})) \otimes \Z'[\tfrac{1}{f}]$
	for $f = |(\mc{O}/\mf{f})^{\times}|$. Since the value ${}_{\mf{c}} \Theta_{i,\mc{N}}(\gamma)$ is independent of 
	the choice of $\mf{f}$ in the sense of Proposition \ref{indep_of_f}, we therefore have that it lies in 
	$K_2(\mc{O}'(\mc{N}))_{\Z'}$.
\end{proof}

In fact, we claim that the class ${}_{\mf{c}} \Theta_{i,\mc{N}}$ is Eisenstein with values in this smaller group. The key point is the following lemma.

\begin{lemma} \label{inj_Gamma_0}
	For any nonzero ideal $\mf{d}$ of $\mc{O}$, the map 
	$$
		H^1(\Gamma_0(\mc{N})_i,K_2(\mc{O}(\mc{N})[\tfrac{1}{\mf{d}}])) \to H^1(\Gamma_0(\mc{N})_i,K_2(F(\mc{N})))
	$$
	is injective.
\end{lemma}

\begin{proof}
	Recall that $H$ denotes the Hilbert class field of $F$. Consider the commutative square
	$$
		\begin{tikzcd}
		K_2(H) \arrow[two heads]{r}{} \arrow{d}{} & \bigoplus_{\mf{q} \nmid \mf{d}} k_{\mf{q}}^{\times} \arrow{d}{\wr} \\
		K_2(F(\mc{N}))^{\Gal(F(\mc{N)}/H)}  \arrow{r}{} & (\bigoplus_{\mf{Q} \nmid \mf{d}} k(\mc{N})_{\mf{Q}}^{\times})^{\Gal(F(\mc{N})/H)},
		\end{tikzcd}
	$$
	where $k_{\mf{q}}$ (resp., $k(\mc{N})_{\mf{Q}}$) denotes the residue field of a prime $\mf{q}$ of $H$ (resp., $\mf{Q}$ of
	$F(\mc{N})$). The upper horizontal arrow is the surjection in the
	standard localization sequence in $K$-theory, and the right-hand vertical arrow is clearly an isomorphism.
	Thus, the lower horizontal map is surjective as well, which yields the desired injectivity, as $\Gamma_0(\mc{N})_i$ acts on the groups in
    question through its surjective image in $\Gal(F(\mc{N})/H)$.
\end{proof}

This yields the Eisenstein property of the integral cocycles ${}_{\mf{c}} \Theta_{i,\mc{N}}$ as a corollary.

\begin{corollary} \label{int_cocyc_Eis}
	The collection ${}_{\mf{c}} \Theta_{\mc{N}} \in \bigoplus_{i \in I^2} 
	H^1(\Gamma_0(\mc{N})_i,K_2(\mc{O}'(\mc{N}))_{\Z'})$
	is Eisenstein away from $\mc{N}$. That is, it is
	annihilated by $T_{\mf{p}} - (N\mf{p} + \mc{R}_{\mc{N}}(\mf{p}))$ for all primes $\mf{p}$ of $\mc{O}$
	not dividing $\mc{N}$.
\end{corollary}

\begin{proof}
	The stated Eisenstein property, but for cohomology with coefficients in $K_2(F(\mc{N}))_{\Z'}$,
	is Corollary \ref{spec_cocyc_Eis} and Proposition \ref{prop_pull_galois}. We then apply Lemma \ref{inj_Gamma_0}.
\end{proof}

\subsection{Unaugmented cocycles} \label{unaugmented}

Let $p \ge 7$ be a prime number. Write $Q = \Gal(F(\mc{N})/F)$ as a product $Q_p \times Q'$ of its Sylow 
$p$-subgroup  $Q_p$ and its maximal prime-to-$p$ order subgroup $Q'$. The maximal ideals of $\zp[Q]$
correspond to $G_{\qp}$-conjugacy classes of $p$-adic characters of $Q'$. 

Fix a $p$-adic character $\chi$ of $Q'$, and let $\mc{O}_{\chi}$ be the $\zp$-algebra generated by its image. The localization of $\zp[Q]$ determined by $\chi$ is isomorphic to $\mc{O}_{\chi}[Q_p]$, with the projection map $\tilde{\chi} \colon \zp[Q] \to \mc{O}_{\chi}[Q_p]$ coming from the $\zp[Q_p]$-linear extension of $\chi$. For a $\zp[Q]$-module $M$, let us define the \emph{$\chi$-component} of $M$ as 
$$
	M^{(\chi)} = M \otimes_{\zp[Q]} \mc{O}_{\chi}[Q_p]
$$ 
where the right tensor product is given by $\tilde{\chi}$. Though defined as a quotient, $M^{(\chi)}$ is also a direct summand of $M$ via the idempotent determined by $\tilde{\chi}$.

In the following, we extend $\chi$ to $Q$ by taking it to be trivial on $Q_p$. We then view it as a character of $\Cl_{\mc{N}}(F)$
via the Artin map $\mc{R}_{\mc{N}}$.

\begin{lemma} 
	Let $\mf{c}$ be an ideal of $\mc{O}$ prime to $\mc{N}$.
	The projection of $N\mf{c}^2 - \mc{R}_{\mc{N}}(\mf{c})$ to $\zp[Q]^{(\chi)}$ is a unit if and only if 
	$\chi(\mf{c}) \not\equiv N\mf{c}^2 \bmod p\mc{O}_{\chi}$. 
\end{lemma}

\begin{proof}
	The element $N\mf{c}^2 - \mc{R}_{\mc{N}}(\mf{c})$ projects to a unit in the $\chi$-component of $\zp[Q]$ if and only if it 
	reduces to a unit in the coinvariant group for $Q_p$, which is isomorphic to $\mc{O}_{\chi}$. 
	Equivalently, $N\mf{c}^2 - \mc{R}_{\mc{N}}(\mf{c})$ is a unit if and only if $N\mf{c}^2-\chi(\mf{c})$ is. As $\mc{O}_{\chi}$
	is an unramified extension of $\zp$, we have the first statement. 
\end{proof}

If all primes over $p$ divide $\mc{N}$, which is to say $(p) \mid \mc{N}^2$, then we let $\omega \colon Q' \to 
\zp^{\times}$ be the composition of restriction to $\Q(\mu_p)$ with the canonical injection 
$\Gal(\Q(\mu_p)/\Q) \hookrightarrow \zp^{\times}$ that is the unique lift of the modulo $p$ cyclotomic character.
The character induced by $\omega$ on $\Cl_{\mc{N}}(F)$ agrees modulo $p$ with the reduction of the norm map. 

In what follows, we shall take $\chi \neq \omega^2$ as being automatically satisfied if $(p) \nmid \mc{N}^2$.

\begin{corollary} \label{factor_unit}
	If $\chi \neq \omega^2$, then there exists an ideal $\mf{c}$ of $\mc{O}$ 
    prime to $\mc{N}$ such that 
	the projection of $N\mf{c}^2 - \mc{R}_{\mc{N}}(\mf{c})$ to $\zp[Q]^{(\chi)}$ is a unit. 
\end{corollary}

Let 
$$
	{}_{\mf{c}} \Theta_{i,\mc{N}}^{\chi} \colon \Gamma_0(\mc{N})_i \to (K_2(\mc{O}'(\mc{N})) \otimes \zp)^{(\chi)}
$$ 
be the cocycle given by composing ${}_{\mf{c}} \Theta_{i,\mc{N}}$ with projection to the $\chi$-component of the $p$-part
of $K_2(\mc{O}'(\mc{N}))$.
 
\begin{theorem} \label{unmodified_eigensp}
 	For all $\chi \neq \omega^2$ and each $i \in I^2$, there is a unique class
	$$
		\Theta_{i,\mc{N}}^{\chi} 
		\in H^1(\Gamma_0(\mc{N})_i,(K_2(\mc{O}'(\mc{N})) \otimes \zp)^{(\chi)}).
	$$
	such that
	\begin{equation} \label{recover_cocyc}
		(N\mf{c}^2 - \mc{R}_{\mc{N}}(\mf{c}))\Theta_{i,\mc{N}}^{\chi} =  {}_{\mf{c}} \Theta_{i,\mc{N}}^{\chi}
	\end{equation}
	for every ideal $\mf{c}$ of $\mc{O}$ prime to $\mc{N}$. 
    Moreover, the collection $\Theta_{\mc{N}}^{\chi} = (\Theta_{i,\mc{N}}^{\chi})_{i \in I^2}$
	is Eisenstein away from $\mc{N}$.
\end{theorem}

\begin{proof}
	We define 
	$$
		\Theta_{i,\mc{N}}^{\chi} = (N\mf{d}^2 - \mc{R}_{\mc{N}}(\mf{d}))^{-1} {}_{\mf{d}} \Theta_{i,\mc{N}}^{\chi}
	$$
	in $H^1(\Gamma_0(\mc{N})_i,(K_2(\mc{O}'(\mc{N})) \otimes \zp)^{(\chi)})$
	for $\mf{d}$ such that $N\mf{d}^2 - \chi(\mf{d})$ is a $p$-adic unit, which exists by Corollary \ref{factor_unit}. The property \eqref{recover_cocyc}
    of $\Theta_{i,\mc{N}}^{\chi}$ follows from Corollary \ref{compare_aux_ideal} and Lemma \ref{inj_Gamma_0}
    and clearly implies uniqueness. 
	The tuple $\Theta_{\mc{N}}^{\chi}$ of classes is Eisenstein away from $\mc{N}$ by Corollary \ref{int_cocyc_Eis}.
\end{proof}

\begin{remark}
	Even for $\chi = \omega^2$, we can make sense of $p(N\mf{c}^2-\mc{R}_{\mc{N}}(\mf{c}))^{-1} {}_{\mf{c}} 
	\Theta_{i,\mc{N}}^{\omega^2}$ for a good choice of $\mf{c}$. That is, what we might denote
	$p\Theta_{i,\mc{N}}^{\omega^2}$ is well-defined, even if it is not clear that $\Theta_{i,\mc{N}}^{\omega^2}$ is.
\end{remark}
 
If we can construct such $\Theta_{i,\mc{N}}^{\chi}$ for all $\chi$ (including $\omega^2$ if $(p) \mid \mc{N}^2$), then
we can sum them to obtain an unaugmented class $\Theta_{i,\mc{N}}$ on $p$-parts, as in the following theorem.
Recall that $h$ is the class number of $F$.

\begin{theorem} \label{unmodified_assumpt}
	Suppose that either $(p) \nmid \mc{N}^2$ or $p \nmid h$.
	For $i \in I^2$, there exists a unique class $\Theta_{i,\mc{N}} \in H^1(\Gamma_0(\mc{N})_i,K_2(\mc{O}'(\mc{N})) \otimes \zp)$ 
	such that
	$$
		(N\mf{c}^2 - \mc{R}_{\mc{N}}(\mf{c})) \Theta_{i,\mc{N}} = {}_\mf{c} \Theta_{i,\mc{N}}
	$$
	for all ideals $\mf{c}$ of $\mc{O}$ prime to $\mc{N}$. The collection $\Theta_{\mc{N}} = (\Theta_{i,\mc{N}})_{i \in I^2}$
	is Eisenstein away from $\mc{N}$.
\end{theorem} 

\begin{proof}
	If $(p) \nmid \mc{N}^2$, then the result follows from Theorem \ref{unmodified_eigensp}, since the condition $\chi \neq \omega^2$
	is automatically satisfied. We therefore suppose that $(p) \mid \mc{N}^2$ and $p \nmid h$, in which case it similarly suffices to show that the $\omega^2$-component of $K_2(\mc{O}'(\mc{N})) \otimes \Z_p$ is trivial.
	
	Note that either $\mc{N}$ is not a prime power or $p$ is non-split in $F$ and $\mc{N}$ is a power of the prime over $p$.
	In the former case, $\mc{O}'(\mc{N}) = \mc{O}(\mc{N})$, and in the latter case
	$K_2(\mc{O}'(\mc{N})) \otimes \Z_p \cong K_2(\mc{O}(\mc{N})) \otimes \Z_p$. 
	Now, by a classical result of Tate \cite[Theorem 5.4]{tate}, we have
	$$
		K_2(\mc{O}(\mc{N})) \otimes \Z_p \cong H^2_{\et}(\mc{O}(\mc{N})[\tfrac{1}{p}],\Z_p(2)).
	$$ 
	
	Let us show that the $\omega^2$-component of the latter cohomology group
	is trivial if $p \nmid h$ by showing that its quotient by the action of the maximal ideal of $\Z_p[Q_p]$ is.
	Since the Galois group of the maximal unramified outside $p$-extension of a number field has $p$-cohomological dimension $2$ (cf.~\cite[Lemma 3.3.11]{nsw}), 
    corestriction gives the first isomorphism (cf.~\cite[Theorem 10.2.3]{nsw}) in 
	$$
		H^2_{\et}(\mc{O}(\mc{N})[\tfrac{1}{p}],\zp(2))^{(\omega^2)} \otimes_{\zp[Q_p]} \F_p
		\cong H^2_{\et}(\mc{O}[\mu_p,\tfrac{1}{p}],\mu_p^{\otimes 2})^{(\omega^2)}
		\cong H^2_{\et}(\mc{O}[\mu_p,\tfrac{1}{p}],\mu_p)^{(\omega)}.
	$$
	Recall that Kummer theory provides an exact sequence
	$$
		0 \to \Cl'(F(\mu_p)) \otimes \F_p\to H^2_{\et}(\mc{O}[\mu_p,\tfrac{1}{p}],\mu_p)
		\to \bigoplus_{v \mid p} \F_p \to \F_p \to 0
	$$
    (cf.~\cite[Proposition 8.3.11]{nsw}), where here $\Cl'(F(\mu_p))$ denotes
    the quotient of $\Cl(F(\mu_p))$ by the classes of primes over $p$.
	As a $\zp[\Gal(F(\mu_p)/\Q)]$-module, the $\omega$-component of this cohomology group breaks up as a sum
	of two components, that for the composition $\omega_{\Q} \colon \Gal(F(\mu_p)/\Q) \to \zp^{\times}$
	of the mod $p$ cyclotomic character with the lift of the reduction mod $p$ map, and that for the product of $\omega_{\Q}$
	and the $p$-adic character $\chi_F$ of the imaginary quadratic field $F$. It follows from a quick
	examination of these groups that
	$H^2_{\et}(\mc{O}[\mu_p,\tfrac{1}{p}],\mu_p)^{(\omega)}$ is just the $\omega_{\Q}\chi_F$-component of 
	$\Cl'(F(\mu_p)) \otimes \F_p$ for the action of $\Gal(F(\mu_p)/\Q)$. By Leopoldt's  Spiegelungssatz, this component is trivial as $\Cl(F) \otimes \F_p$ is.
\end{proof}

Quite frequently, then, we have that the conditions of the following corollary are satisfied. In fact, the above
proof shows that the weaker condition of the triviality of the $\omega_{\Q}\chi_F$-component of 
$\Cl'(F(\mu_p)) \otimes \F_p$ can be used to replace the condition of $p$ dividing $h$.

\begin{corollary}
	Suppose that there are no primes greater than $5$ 
    dividing both $\mc{N}^2$ and $h$. For $i \in I^2$, there exists
	a unique class $\Theta_{i,\mc{N}} \in H^1(\Gamma_0(\mc{N})_i,K_2(\mc{O}'(\mc{N}))_{\Z'})$ such that
	$$
		(N\mf{c}^2 - \mc{R}_{\mc{N}}(\mf{c})) \Theta_{i,\mc{N}} = {}_\mf{c} \Theta_{i,\mc{N}}
	$$
	for all ideals $\mf{c}$ of $\mc{O}$ prime to $\mc{N}$. The collection $\Theta_{\mc{N}} = (\Theta_{i,\mc{N}})_{i \in I^2}$
	is Eisenstein away from $\mc{N}$.
\end{corollary}

\section{Eisenstein maps on homology}

\subsection{Cohomology of Bianchi spaces} \label{cohom_Bianchi}

Let us quickly review the discussion of Section \ref{Hecke_op_corr} in our setting of interest.
Let $\mb{H} = \mb{H}_{2,F} = \C \times \R_{>0}$ denote the complex upper half-space,
with the usual action of $\GL_2(F)$. 
We view $I$ as the subset of $I^2$ consisting of pairs $(r,1)$ for $r \in I$.
For $r \in I$, we then have $\Gamma_1(\mc{N})_r = \Gamma_1(\mc{N})_{(r,1)}$ and similarly for other subscripts.
The Bianchi space for $F$ of level $\Gamma_1(\mc{N})$ is the disjoint union $Y_1(\mc{N}) = \coprod_{r \in I} Y_1(\mc{N})_r$,
where $Y_1(\mc{N})_r = \Gamma_1(\mc{N})_r \backslash \mb{H}$. 
Then $Y_1(\mc{N})$ also has the usual adelic description
$$
	Y_1(\mc{N}) = \GL_2(F) \backslash (\GL_2(\mb{A}_F^f) \times \mb{H}) / U_1(\mc{N}).
$$

Any element of finite order in $\Gamma_1(\mc{N})_r$ has order dividing $120$, in that its eigenvalues are roots of unity with sum and product in $F$, which are therefore contained in a number field of degree $4$.
As in \eqref{isom_cohom}, for any $\Delta_0(\mc{N})$-module system $A = (A_r)_{r \in I}$ indexed by $I$ such that $5!$ acts invertibly and scalar elements act trivially on each $A_r$, we have a canonical isomorphism
\begin{equation} \label{Bianchi_cohom}
	H^1(Y_1(\mc{N}),A) \cong \bigoplus_{r \in I} H^1(\Gamma_1(\mc{N})_r,A_r).
\end{equation}
This isomorphism is Hecke equivariant for the operators defined in and following 
Definition \ref{more_gen_Hecke} by Proposition \ref{Hecke_gp_top}. On the right, 
if $\mf{a}_{r,s}\mf{n}$ is principal, then we have $T_{\mf{n}} = T_{r,s}(\mf{n},1)$ by Lemma \ref{different_Hecke}(b). 
If $\mf{a}_{r,s}\mf{n}^2$ is principal, then we set $S_{\mf{n}} = T_{r,s}(\mf{n},\mf{n})$. 
Let us write $\langle \mf{n} \rangle$ for $(\langle \mf{n} \rangle^*)^{-1}$.
By Lemma \ref{different_Hecke}(a),
we have $S_{\mf{n}} = \langle \mf{n} \rangle [\mf{n}]^* = [\mf{n}]^* \langle \mf{n} \rangle$.

If the action of $\Gamma_1(\mc{N})_r$ on $A_r$ is trivial, then we have a canonical isomorphism
$$
	\phi_r \colon H^1(\Gamma_1(\mc{N})_r,A_r) \xrightarrow{\sim} \Hom(H_1(\Gamma_1(\mc{N})_r,\Z'),A_r).
$$
Suppose that the $A_r$ are all equal to a fixed $A$ so that the maps $\phi_r$ assemble into an isomorphism
$$
	\phi \colon H^1(Y_1(\mc{N}),A) \xrightarrow{\sim} \Hom(H_1(Y_1(\mc{N}),\Z'),A).
$$
Suppose also that
each $g = \smatrix{a&b\\c&d} \in \Delta_0(\mc{N})_{r,s}$ provides a map $g \colon A \to A$ depending only on the image of $d$ in $(\mc{O}/\mc{N})^{\times}$, independent of $r,s \in I$.

Given $g \in \Delta_0(\mc{N})_{r,s}$, and writing $\Gamma_1(\mc{N})_r g \Gamma_1(\mc{N})_s = \coprod_{t=1}^v g_t\Gamma_1(\mc{N})_s$, we define $T(g)$
on $x \in H_1(Y_1(\mc{N})_r,\Z')$ by 
$$
	T(g)x = \sum_{t=1}^v g_t^{\dagger}x \in H_1(Y_1(\mc{N})_s,\Z'),
$$ 
where $g_t^{\dagger}$ denotes the adjoint matrix to $g_t$. We then have
the following identity, as the reader may verify by a similar argument to that given in the proof of \cite[Theorem 4.3.7]{sv}.

\begin{lemma} \label{convert_cohom_hom}
    For $\xi \in H^1(Y_1(\mc{N})_s,A)$ and $x \in H_1(Y_1(\mc{N})_r,\Z')$, we have
    $$
	   \phi(T(g)\xi)(x) = g \cdot \phi(\xi)(T(g)x).
    $$
\end{lemma}

When $g$ is such that $T(g) \colon H^1(Y_1(\mc{N})_s,A) \to H^1(Y_1(\mc{N})_r,A)$ is the restriction of $T_{\mf{n}}$, $S_{\mf{n}}$, or $\langle \mf{n} \rangle$
for some ideal $\mf{n}$ of $\mc{O}$ prime to $\mc{N}$, we write $T_{\mf{n}}$, $S_{\mf{n}}$, or $\langle \mf{n} \rangle$,
respectively, for the operator $H_1(Y_1(\mc{N})_r,\Z') \to H_1(Y_1(\mc{N})_s,\Z')$ given by $T(g)$ and also for the corresponding operator on $H_1(Y_1(\mc{N}),\Z')$. 

Let us consider the restriction of ${}_{\mf{c}} \Theta_{r,\mc{N}}^{\chi}$ to $\Gamma_1(\mc{N})_r$. Note that the above conditions on $A = K_2(\mc{O}'(\mc{N}))_{\Z'}$ are satisfied, and
$g = \smatrix{a&b\\c&d} \in \Delta_0(\mc{N})_{r,s}$ acts as $\sigma_d$ on $K_2(\mc{O}'(\mc{N}))$.
The restriction in question yields a homomorphism
$$
	{}_{\mf{c}} \Pi_{r,\mc{N}} \colon H_1(Y_1(\mc{N})_r,\Z') \to K_2(\mc{O}'(\mc{N}))_{\Z'}.
$$
We can then take the sum
$$
	{}_{\mf{c}} \Pi_{\mc{N}} \colon H_1(Y_1(\mc{N}),\Z') \to K_2(\mc{O}'(\mc{N}))_{\Z'}
$$
of these maps over $r \in I$. 

For a prime $p \ge 7$ and character $\chi \neq \omega^2$ as in Section \ref{unaugmented}, we also have maps
\begin{equation} \label{Pi_on_homology}
	\Pi_{r,\mc{N}}^{\chi} \colon H_1(Y_1(\mc{N})_r,\zp) \to (K_2(\mc{O}'(\mc{N})) \otimes \zp)^{(\chi)}
\end{equation}
and $\Pi_{\mc{N}}^{\chi}$, and the analogues of the statements which follow also hold for these.

\begin{proposition} \label{equivar_maps}
	For $r,s \in I$ and a nonzero ideal $\mf{n}$ of $\mc{O}$ prime to $\mc{N}$ such that 
	$\mf{a}_{r,s}\mf{n}^2$ is principal, we have 
	$$
		{}_{\mf{c}} \Pi_{s,\mc{N}} \circ S_{\mf{n}} = \mc{R}_{\mc{N}}(\mf{n}) \circ {}_{\mf{c}} \Pi_{r,\mc{N}}.
	$$
	In particular, for $d \in (\mc{O}/\mc{N})^{\times}$, we have
	$$
		{}_{\mf{c}} \Pi_{r,\mc{N}} \circ \langle d \rangle = \mc{R}_{\mc{N}}(d) \circ {}_{\mf{c}} \Pi_{r,\mc{N}}.
	$$
\end{proposition}

\begin{proof}
	Let $q \in I$ be such that $\mf{a}_{r,q}\mf{n}$ and $\mf{a}_{s,q}\mf{n}^{-1}$ are principal.
	Let $u \in I$ be such that $\mf{a}_u\mf{n}$ is principal.
	Lemma \ref{different_Hecke}(a), Corollary \ref{cocyc_fixed_by_diamond}, and Lemma \ref{prop_pull_galois} tell us that
	$$
		S_{\mf{n}}({}_{\mf{c}} \Theta_{s,\mc{N}}) = [\mf{n}]^* \langle \mf{n} \rangle ({}_{\mf{c}} \Theta_{s,\mc{N}})
		= [\mf{n}]^* ({}_{\mf{c}} \Theta_{(q,u),\mc{N}}) = \mc{R}_{\mc{N}}(\mf{n}) \circ \Theta_{r,\mc{N}}.
	$$
	On the other hand, note that the operator $S_{\mf{n}} = [\mf{n}]^* \langle \mf{n} \rangle$ on cohomology is given by pullback
	by a matrix in $\Delta_0(\mc{N})_{r,s}$ with lower right-hand entry congruent to a unit modulo $\mc{N}$.
	We then need only apply Lemma \ref{convert_cohom_hom} for $\xi$ corresponding to ${}_{\mf{c}} \Theta_{s,\mc{N}}$ and $T(g) = S_{\mf{n}}$ to obtain
	the first statement. The last statement follows by observing that $[d]^*$, the pullback by a scalar matrix, 
	acts trivially on $H_1(Y_1(\mc{N})_r,\Z')$.
\end{proof}

The Eisenstein property of the maps ${}_{\mf{c}} \Pi_{\mc{N}}$ is a consequence of Corollary \ref{spec_cocyc_Eis}: that is,
\begin{equation} \label{Eis_prop}
	{}_{\mf{c}} \Pi_{\mc{N}} \circ (T_{\mf{p}} - N\mf{p} - S_{\mf{p}}) = 0
\end{equation}
for every prime ideal $\mf{p}$ of $\mc{O}$ not dividing $\mc{N}$. 
Together with Proposition \ref{equivar_maps}, we may rephrase this as the following.

\begin{proposition} \label{Eis_Galois}
	For every prime ideal $\mf{p}$ of $\mc{O}$ not dividing $\mc{N}$, we have
	$$
		{}_{\mf{c}} \Pi_{\mc{N}} \circ T_{\mf{p}} = (N\mf{p} + \mc{R}_{\mc{N}}(\mf{p})) \circ {}_{\mf{c}} \Pi_{\mc{N}}.
	$$
	More precisely, 
	$$
		{}_{\mf{c}} \Pi_{s,\mc{N}} \circ T_{\mf{p}} = (N\mf{p} + \mc{R}_{\mc{N}}(\mf{p})) \circ {}_{\mf{c}} \Pi_{r,\mc{N}}
	$$
	for $r, s \in I$ such that $\mf{a}_{r,s}\mf{p}$ is principal.
\end{proposition}

We expect that the maps on the homology of $Y_1(\mc{N})$ factor through the homology of the compactification $X_1(\mc{N})$ given by adjoining cusps. We prove this for certain $\chi$ in the next subsection.

\subsection{Borel-Serre boundary}

The Borel-Serre compactification $X_1^{\mr{BS}}(\mc{N})$ of $Y_1(\mc{N})$ can be written as 
$$
	X_1^{\mr{BS}}(\mc{N}) = \GL_2(F) \backslash (\GL_2(\mb{A}_F^f) \times \overline{\mb{H}}) / U_1(\mc{N})
$$
where $\overline{\mb{H}} = \mb{H} \cup \partial\overline{\mb{H}}$ with
$\partial\overline{\mb{H}}$ a disjoint union of components indexed by $\alpha \in \mb{P}^1(F)$ that can be thought of as placing
$(0,1] \times (\mb{P}^1(\C)-\{\alpha\})$ at $\alpha$. Here, an element of $\GL_2(F)$ acts on $(\alpha,t,x)$ for $t \in (0,1)$ and 
$x \in \mb{P}^1(\C)$ as a M\"obius transformation on the first and last coordinates. The $r$th component is $X_1^{\mr{BS}}(\mc{N})_r = \Gamma_1(\mc{N})_r \backslash \overline{\mb{H}}$. The embedding of $Y_1(\mc{N})$ in 
$X_1^{\mr{BS}}(\mc{N})$ is a homotopy equivalence.

For $B$ the upper-triangular Borel subgroup of $G = \GL_2(F)$,
we have $G/B \cong \mb{P}^1(F)$ via the $G$-equivariant map sending the identity to $\infty$. Given $x \in \mb{P}^1(F)$ and any $\tau_x \in G$ mapping to $x$, the group $B_{x,r} = \Gamma_1(\mc{N})_r \cap \tau_x B \tau_x^{-1}$ is the stabilizer of $x$ in $\Gamma_1(\mc{N})_r$. Letting $C_1(\mc{N})_r = \Gamma_1(\mc{N})_r \backslash \mb{P}^1(F)$, we also use $B_{x,r}$ for $x \in C_1(\mc{N})_r$ to denote $B_{\tilde{x},r}$ for a choice of $\tilde{x} \in \mb{P}^1(F)$ lifting the cusp $x$.
As in the discussion of \cite[p.~48--49]{harder-per1} (see also \cite[p.~20]{berger}), the space
\begin{equation} \label{BShomotopy}
	\coprod_{ x \in C_1(\mc{N})_r } B_{x,r} \backslash \mb{H}
\end{equation}
is homotopy equivalent to $\partial X_1^{\BS}(\mc{N})_r$.
This yields the following description 
of the singular first cohomology of 
the Borel-Serre boundary:
\begin{equation} \label{cohom_bound}
    H^1(\partial X_1^{\mr{BS}}(\mc{N}),\Z')
    \cong \bigoplus_{r \in I} \bigoplus_{x \in C_1(\mc{N})_r} H^1(B_{x,r},\Z').
\end{equation}
In particular, $H^1(\partial X_1^{\mr{BS}}(\mc{N}),\Z')$ is torsion-free.

For a $\Delta$-module system $(A_r)_{r \in I}$, we can define the sheaf $A$ on $X_1^{\mr{BS}}(\mc{N})$ 
as we did on $Y_1(\mc{N})$, and this agrees with the pushforward sheaf from the latter space. We then have Hecke actions on the cohomology of $X_1^{\mr{BS}}(\mc{N})$ and $\partial X_1^{\mr{BS}}(\mc{N})$ with $A$-coefficients, and the maps on cohomology induced by pullback of the embedding of $Y_1(\mc{N})$ into
$X_1^{\mr{BS}}(\mc{N})$ are equivariant for the actions of Hecke operators. 

The first interior cohomology group $H^1_!(Y_1(\mc{N}),A)$
of $Y_1(\mc{N})$ with $A$-coefficients is the image of the second map,
or kernel of the third, in the exact sequence
\begin{equation} \label{ex_seq_cpt_supp}
	H^0(\partial X_1^{\BS}(\mc{N}),A) \to H^1_c(Y_1(\mc{N}),A) \to H^1(Y_1(\mc{N}),A) \to H^1(\partial X_1^{\mr{BS}}(\mc{N}), A).
\end{equation}
By \eqref{cohom_bound} and the fact that $\Gamma_1(\mc{N})_r$ has no elements of prime order at least $7$, this interior cohomology group can be identified with the sum over $r \in I$ of the parabolic cohomology groups
$$
	H^1_P(\Gamma_1(\mc{N})_r,A) = \ker\Bigl(H^1(\Gamma_1(\mc{N})_r,A) \to \bigoplus_P H^1(P,A)\Bigr),
$$
where $P$ runs over (representatives of $\Gamma_1(\mc{N})_r$-conjugacy classes of)  parabolic subgroups of $\Gamma_1(\mc{N})_r$, or equivalently, stabilizers $B_{x,r}$ of chosen lifts of cusps $x \in C_1(\mc{N})_r$ on the $r$th component $X_1(\mc{N})_r$ of the Satake compactification $X_1(\mc{N})$ of $Y_1(\mc{N})$. In fact, we have the following.

\begin{lemma}
    We have canonical isomorphisms
    $$
        \Hom(H_1(X_1(\mc{N}),\Z'),A) \cong H^1_!(Y_1(\mc{N}),A)
        \cong \bigoplus_{r \in I} H^1_P(\Gamma_1(\mc{N})_r,A),
    $$
    compatible with Hecke actions.
\end{lemma}

\begin{proof}
    Let $C_1(\mc{N}) = \coprod_{r \in I} C_1(\mc{N})_r$ denote the zero-dimensional
    space of cusps on $X_1(\mc{N})$. Poincar\'e duality and the universal
    coefficient theorem (again using that $30$ is invertible in $\Z'$) yield canonical isomorphisms
    $$
    \Hom(H_1(X_1^{\BS}(\mc{N}),\partial X_1^{\BS}(\mc{N}),\Z'),A) \cong \Hom(H^2(Y_1(\mc{N}),\Z'),A) \cong H^1_c(Y_1(\mc{N}),A).
    $$
    This allows us to identify the first map in the exact sequence
    \eqref{ex_seq_cpt_supp} with the upper horizontal map in the commutative square
    $$
    \begin{tikzcd}
        \Hom(H_0(\partial X_1^{\BS}(\mc{N}),\Z'),A) \arrow{r} \arrow{d}{\wr} &
        \Hom(H_1(X_1^{\BS}(\mc{N}),\partial X_1^{\BS}(\mc{N}),\Z'),A) \arrow{d}{\wr} \\
        \Hom(H_0(C_1(\mc{N}),\Z'),A) \arrow{r} & \Hom(H_1(X_1(\mc{N}),C_1(\mc{N}),\Z'),A).
    \end{tikzcd}
    $$
    The cokernel of the lower horizontal map is
    $\Hom(H_1(X_1(\mc{N}),\Z'),A)$,
    which gives the first isomorphism, the second already having been explained.
    The maps in question are Hecke-equivariant by construction and our earlier
    discussions.
\end{proof}

Let $I_{\mc{N}}$ be the prime-to-$\mc{N}$ ideal group of $F$. We use the convention that a Hecke character of conductor $\mf{m}$
dividing $\mc{N}$ and infinity type $(j,k)$ with $j,k \in \Z$ is a homomorphism $\phi \colon I_{\mc{N}} \to \C^{\times}$ such that for $x \in \mc{O}$ with $x \equiv 1 \bmod \mf{m}$, we have $\phi((x)) = x^j \bar{x}^k$, and $\mf{m}$ is the largest ideal with this property.
Such a $\phi$ gives rise to a map $\tilde{\phi} \colon \mb{A}_F^{\times}/F^{\times} \to \C^{\times}$ such that the restriction 
of $\tilde{\phi}$ to the finite ideles induces $\phi$ and which at the infinite place is just
$\C^{\times} \to \C^{\times}$ by $z \mapsto z^{-j} \bar{z}^{-k}$.

Let $Z$ (resp. $Z'$) denote the set of pairs $\mu = (\mu_1,\mu_2)$
of Hecke characters $I_{\mc{N}} \to \C^{\times}$ of respective infinity types
$(-1,0)$ and $(1,0)$ (resp. $(0,-1)$ and $(0, 1)$)
such that the product of the conductors $\mf{f}_1$ and $\mf{f}_2$ 
of $\mu_1$ and $\mu_2$ divides $\mc{N}$. 
We view $\mu \in Z\cup Z'$ as a map $B(F) \backslash B(\mb{A}_F)\to \C^{\times}$ that sends the class of a matrix $\smatrix{a&b\\0&d}$ to
$\mu_1(a)\mu_2(d)$.

For a commutative ring $C$, let $\mf{H}_C(\mc{N})$ denote the Hecke $C$-algebra of endomorphisms of $H^1(Y_1(\mc{N}),C) \oplus H^1(\partial X_1^{\mr{BS}}(N),C)$ generated by the operators $S_{\mf{p}}$ and $T_{\mf{p}}$
for primes $\mf{p}$ of $\mc{O}$ not dividing $\mc{N}$. 
For $\mu \in Z$ and an $\mf{H}_{\C}(\mc{N})$-module $M$, let $M_{\mu}$ denote the maximal submodule of $A$ upon which each $S_{\mf{p}}$ acts as $\mu_1\mu_2(\mf{p})$ and each $T_{\mf{p}}$ acts as $\mu_1(\mf{p})N\mf{p} + \mu_2(\mf{p})$.

\begin{lemma} \label{dir_sum_boundary_cohom_C}
	We have a direct sum decomposition
	$$
		H^1(\partial X_1^{\mr{BS}}(\mc{N}),\C) \cong \bigoplus_{\mu \in Z} H^1(\partial X_1^{\mr{BS}}(\mc{N}),\C)_{\mu}
	$$ 
	of $\mf{H}_{\C}(\mc{N})$-modules.
\end{lemma}

\begin{proof}
	For $\mu \in Z\cup Z'$, let $V(\mu)$ denote the $\mf{H}_{\C}(N)$-module of right $U_1(N)$-invariant maps 
	$\psi \colon \GL_2(\mb{A}_F^f) \to \C$ 
	such that $\psi(bg) = \mu(b)\psi(g)$ for all $b \in B(\mb{A}_F^f)$ and $g \in \mc{G}$.
	By \cite[Theorem 1]{harder} (see also the discussion of \cite[Section~2.10.1]{berger}), the isomorphism of \eqref{cohom_bound} gives rise to
	an isomorphism of $\mf{H}_{\C}(N)$-modules
	$$
		H^1(\partial X_1^{\mr{BS}}(\mc{N}),\C) \cong \bigoplus_{\mu \in Z} (V(\mu) \oplus V(\mu')),
	$$
	where $\mu' = (\mu_2 N^{-1},\mu_1 N) \in Z'$. Here, $N$ is the absolute norm,
	which is a Hecke character of type $(1,1)$ having trivial conductor.
	Now, $V(\mu)$ decomposes as a restricted tensor product
	$V(\mu) = \bigotimes'_{\mf{p}} V_{\mf{p}}(\mu)$ over the primes $\mf{p}$ of $\mc{O}$, and
	the action of $T_{\mf{p}}$ is trivial on all components but $V_{\mf{p}}(\mu)$.
	For $\mf{p} \nmid \mc{N}$,
	the latter representation is one-dimensional (being the right $\GL_2(\mc{O}_{\mf{p}})$-invariants
	of an unramified principal series). 
	For the unique function $\psi_{\mf{p}} \colon \GL_2(F_{\mf{p}}) \to \C$
	in $V_{\mf{p}}(\mu)$ sending the identity to $1$ (i.e., the spherical vector) and 
	for $\pi$ a uniformizer of the valuation ring of $F_{\mf{p}}$, we compute 
	$$
		(T_{\mf{p}} \psi_{\mf{p}})(1) = \psi_{\mf{p}}\left(\smatrix{1 & 0 \\ 0 & \pi}\right)  + \sum_{\bar{b} \in \mc{O}/\mf{p}} \psi_{\mf{p}}
		\left(\smatrix{\pi & b \\ 0 & 1}\right)
		= N\mf{p}\mu_1(\mf{p}) + \mu_2(\mf{p}),
	$$
	so $T_{\mf{p}}$ acts as $N\mf{p}\mu_1(\mf{p}) + \mu_2(\mf{p})$ on $V(\mu)$.
	Similarly, $S_{\mf{p}}$ acts as $\mu_1(\mf{p})\mu_2(\mf{p})$. The complex conjugate of $\mu' = (\mu_1',\mu_2')$ is in $Z$, 
	and $V(\mu')$ has the same Hecke action as $V(\mu)$.
\end{proof}

Let $p$ be an odd prime. Fix a prime $\mc{P}$ over $p$ of the integer ring $\mc{R}$ of the number field generated over $F$ by the images of all $\mu_1$ and $\mu_2$ for $(\mu_1,\mu_2) \in Z$. Let $\mf{p}= \mc{P} \cap \mc{O}$. The restrictions to $I_{\mc{N}\mf{p}}$ of $\mu_1$ and $\mu_2$ take image in the units at $\mc{P}$, so we may speak of their reductions modulo $\mc{P}$. 
 
Let $Z_p$ denote the set of pairs $\nu = (\nu_1,\nu_2)$ of Hecke characters $\nu_1, \nu_2 \colon I_{\mc{N}\mf{p}} \to \F_q^{\times}$ for 
some $p$-power $q$ that are the respective reductions of $\mu_1$ and $\mu_2$ modulo $\mc{P}$ for some $\mu = (\mu_1,\mu_2) \in Z$. Note that $\nu_1$ and $\nu_2$ factor through the quotient $\Cl_{\mc{N} \cap \mf{p}}(F)$, and in fact through its prime-to-$p$ part. The conductor of $\nu_k$ for $k \in \{1,2\}$ divides $\mf{f}_k \cap \mf{p}$ for all choices of $\mu \in Z$ reducing to $\nu$.

\begin{remark}
For any prime ideal $\mf{q}$ of $\mc{O}$, the group $(\mc{O}/\mf{q})^{\times}$ 
maps canonically to $\Cl_{\mc{N} \cap \mf{p} \cap \mf{q}}(F)$, which in turn
has $\Cl_{\mc{N} \cap \mf{p}}(F)$ as a canonical quotient. Thus, we can make sense of the restriction of $\nu_k$ to $(\mc{O}/\mf{q})^{\times}$.
\end{remark}

Let $W$ denote the Witt vectors of $\overline{\F}_p$. We can 
view each $\nu_k$ for $(\nu_1,\nu_2) \in Z_p$ as a character valued in $W^{\times}$ by taking its unique lift.
With this convention, for an $\mf{H}_W(\mc{N})$-module $A$, let $A_{\nu}$ denote its localization
at the maximal ideal containing $S_{\mf{q}} - \nu_1\nu_2(\mf{q})$ and $T_{\mf{q}} - (N\mf{q}\nu_1(\mf{q}) +  \nu_2(\mf{q}))$
for each prime $\mf{q}$ of $\mc{O}$ not dividing $\mc{N}p$.

\begin{proposition} \label{boundary_witt_vecs}
	There exists an isomorphism
	$$
		H^1(\partial X_1^{\mr{BS}}(\mc{N}),W) \cong 
		\bigoplus_{\nu \in Z_p} H^1(\partial X_1^{\mr{BS}}(\mc{N}),W)_{\nu},
	$$
	and similarly with $W$ replaced with $\overline{\F}_p$.
\end{proposition}

\begin{proof}
	Fix an embedding of $\overline{\Q}_p$ into $\C$.
	For $\nu \in Z_p$, let $M_{\nu} = H^1(\partial X_1^{\mr{BS}}(\mc{N}),W)_{\nu}$. This contains the intersection with 
	$M = H^1(\partial X_1^{\mr{BS}}(\mc{N}),W)$ 
	of the direct sum of the $H^1(\partial X_1^{\mr{BS}}(\mc{N}),\C)_{\mu}$ over $\mu \in Z$ reducing to $\nu$.
	On the other hand, the sum $\bigoplus_{\nu \in Z_p} M_{\nu}$ of $W$-modules is a direct summand of $M$, since
	the maximal ideals giving the localizations are distinct. 
	This sum contains $M$ by Lemma \ref{dir_sum_boundary_cohom_C}, so we are done.
\end{proof}

Let $\xi \colon (\mc{O}/\mf{p})^{\times} \to (\mc{R}/\mc{P})^{\times}$ be the canonical inclusion. If $p$ splits in $\mc{O}$, let $\bar{\xi} \colon (\mc{O}/\bar{\mf{p}})^{\times} \to (\mc{R}/\mc{P})^{\times}$ denote precomposition of $\xi$ with complex conjugation. 

\begin{remark} \label{reducechar}
	If $\mu \in Z$ has reduction $\nu \in Z_p$ modulo $\mc{P}$, then $\nu_1(x) = x^{-1} \bmod \mc{P}$ and 
	$\nu_2(x) = x \bmod \mc{P}$ for $x \equiv 1 \bmod \mc{N}$ with $x$ prime to $\mf{p}$. From this, we obtain the following.
	\begin{enumerate}
		\item[i.] If $\mf{p}$ does not divide $\mf{f}_1$ 
		(resp., $\mf{f}_2$) for some choice of $\mu \in Z$ reducing to $\nu$,
		then $\nu_1|_{(\mc{O}/\mf{p})^{\times}} = \xi^{-1}$ (resp., $\nu_2|_{(\mc{O}/\mf{p})^{\times}} = \xi$).
		\item[ii.] Suppose that $p$ splits in $\mc{O}$. If $\bar{\mf{p}}$ does not divide $\mf{f}_k$ for some $k \in \{1,2\}$ and some 
		(equivalently, all) $\mu$ reducing to $\nu$  then $\nu_k|_{(\mc{O}/\bar{\mf{p}})^{\times}} = 1$.
	\end{enumerate}
    The converse to (i) (resp., (ii)) holds if $\mc{N}$ is divisible by
    at most a single power of $\mf{p}$ (resp., $\bar{\mf{p}}$).
\end{remark}

\begin{lemma} \label{condcond}
	Suppose that $\mu \in Z$ is such that $\mu_1\mu_2$ is primitive at some prime $\mf{q}$ dividing $\mc{N}$.
	Then $\mf{q}$ cannot divide both $\mf{f}_1$ and $\mf{f}_2$.
\end{lemma}

\begin{proof}
	The conductor of $\mu_1\mu_2$ divides $\mf{f}_1 \cap \mf{f}_2$, so if neither $\mu_1$ nor $\mu_2$ is primitive at $\mf{q}$,
	then $\mu_1\mu_2$ cannot be either. But if one of $\mu_1$ or $\mu_2$ is primitive at $\mf{q}$, the fact that
	$\mf{f}_1 \mf{f}_2$ divides $\mc{N}$ forces the other to have conductor prime to $\mf{q}$.
\end{proof}

Note that Lemma \ref{condcond} implies that the first homology of the Borel-Serre boundary vanishes in prime level, as there is no Hecke character
of type $(\pm 1,0)$ and trivial conductor. This can also be seen using the fact that, in this case, each Borel subgroup contains
an involution which acts by conjugation as $-1$ on the unipotent subgroup.

\begin{lemma} \label{char_forced}
    Let $\mu \in Z$ reduce to $\nu \in Z_p$ modulo $\mc{P}$.
	\begin{enumerate}
		\item[a.] If $\mu_1\mu_2$ is primitive at $\mf{p}$, then 
		either $\nu_1|_{(\mc{O}/\mf{p})^{\times}} = \xi^{-1}$ or $\nu_2|_{(\mc{O}/\mf{p})^{\times}} = \xi$. 
		\item[b.] 
		Suppose that $p$ is split in $\mc{O}$. If $\mu_1\mu_2$ is primitive at $\bar{\mf{p}}$,
		then either $\nu_1|_{(\mc{O}/\bar{\mf{p}})^{\times}} = 1$ or $\nu_2|_{(\mc{O}/\bar{\mf{p}})^{\times}} = 1$.
	\end{enumerate}
    Moreover, we need not assume the primitivity of $\mu_1\mu_2$ in (a) (resp., (b))
    if $\mc{N}$ is divisible by at most a single power of $\mf{p}$ (resp., $\bar{\mf{p}}$).
\end{lemma}

\begin{proof}
	If the conclusion of part (a) (resp., (b)) fails for $\nu$, then $\mf{p}$ (resp., $\bar{\mf{p}}$) divides both $\mf{f}_1$ and $\mf{f}_2$
    by Remark \ref{reducechar}. Then Lemma \ref{condcond} tells us that 
	$\mu_1\mu_2$ cannot be primitive at $\mf{p}$ (resp., $\bar{\mf{p}}$).
    If $\mc{N}$ is not divisible by the square of $\mf{p}$ (resp., $\bar{\mf{p}}$),
    the assumed failure already gives contradiction of the final statement by the final statement of Remark \ref{reducechar}.
\end{proof}

Let $I$ denote the Eisenstein ideal of the Hecke algebra $\mf{H} = \mf{H}_{\Z'}(\mc{N})$ generated by the 
elements $T_{\mf{q}} - (N\mf{q} + S_{\mf{q}})$ for primes $\mf{q} \nmid \mc{N}$. For an $\mf{H}$-module $M$, let us write $M_{\mr{Eis}}$ to denote the direct sum of the localizations of $M$ at the maximal ideals of $\mf{H}$ containing $I$.

Let $\chi \colon Q' \to \mc{O}_{\chi}^{\times}$ be a character of the prime-to-$p$ part $Q'$ of
$\Cl_{\mc{N}}(F)$. The conductor of $\chi$, extended to $\Cl_{\mc{N}}(F)$ by the trivial character, then automatically has $p$-part dividing the product of the primes of $\mc{O}$ over $p$.

\begin{proposition} \label{Eis_part_BS}
	Suppose that $\mc{N}$ is divisible by at most one power of each prime over $p$, as well as the following:
	\begin{enumerate}
		\item[i.] If $p$ is split in $\mc{O}$, then $\chi|_{(\mc{O}/\mf{p})^{\times}} \not\equiv \xi \bmod \mc{P}$ and
		$\chi|_{(\mc{O}/\bar{\mf{p}})^{\times}} \not\equiv \bar{\xi} \bmod \mc{P}$.
		\item[ii.] If $p$ is inert in $\mc{O}$, then $\chi|_{(\mc{O}/p)^{\times}} \not\equiv \xi, \xi^p \bmod \mc{P}$.
		\item[iii.] If $p$ is ramified in $\mc{O}$, then $\chi|_{(\mc{O}/\mf{p})^{\times}} \not\equiv \xi \bmod \mc{P}$.
	\end{enumerate}
	Then the group $H^1(\partial X_1^{\mr{BS}}(\mc{N}), \zp)^{(\chi)}_{\mr{Eis}}$ is trivial.
\end{proposition}

\begin{proof}
    By Proposition \ref{boundary_witt_vecs}, it suffices to show that if
    $\nu \in Z_p$ with $\nu_1\nu_2 =\chi \bmod \mc{P}$, then there exists some $\mf{q}$
    not dividing $\mc{N}p$ such that
    $\omega(\mf{q}) + \nu_1\nu_2(\mf{q})$ and $\omega\nu_1(\mf{q}) + \nu_2(\mf{q})$ differ in $\mc{R}/\mc{P}$. If this fails, then linear independence
	of characters forces $\nu_1 = 1$ or $\nu_2 = \omega$. 
	If $\nu_1 = 1$, then part (a) of Lemma \ref{char_forced} (in its stronger form
    given by the final statement of the lemma) tells us that
	$\nu_2|_{(\mc{O}/\mf{p})^{\times}} = \xi$, so $\chi_2|_{(\mc{O}/\mf{p})^{\times}} \equiv \xi \bmod \mc{P}$.
    
	Now suppose that $\nu_2 = \omega$. Since $\omega(x) = x\bar{x}$ for $x \in \mc{O}$ prime to $\mc{N}\mf{p}$, the reduction of $\omega|_{(\mc{O}/\bar{\mf{p}})^{\times}}$ modulo $\mc{P}$ equals $\bar{\xi}$ if $p$ is split, 
	$\xi^2$ is $p$ is ramified and $\xi^{p+1}$ if $p$ is inert. 
	If $p$ is split, part (b) of Lemma \ref{char_forced} then forces $\nu_1|_{(\mc{O}/\bar{\mf{p}})^{\times}} = 1$, whereas
	if $p$ is inert or ramified, part (a) of said lemma tells us that $\nu_1|_{(\mc{O}/\mf{p})^{\times}} = \xi^{-1}$. Putting this together, 
	we obtain that $\chi|_{(\mc{O}/\bar{\mf{p}})^{\times}}$ modulo $\mc{P}$ is $\bar{\xi}$ if $p$ is split, $\xi$ if $p$ is ramified,
	and $\xi^p$ if $p$ is inert. 
\end{proof}

Identifying $\Cl_{\mc{N}}(F)$ with $\Gal(F(\mc{N})/F)$ via the Artin map $\mc{R}$, we may conclude the following.

\begin{theorem} \label{Eis_map_homol}
	Let $\mc{N}$, $p \ge 7$, and $\chi \neq \omega^2$ also satisfy the conditions of Proposition \ref{Eis_part_BS}. Then we have a 
	homomorphism
	$$
		\Pi_{\mc{N}}^{\chi} \colon H_1(X_1(\mc{N}),\zp)^{(\chi)} \to (K_2(\mc{O}'(\mc{N})) \otimes \zp)^{(\chi)},
	$$
	compatible with the likewise denoted map on $H_1(Y_1(\mc{N}),\zp)^{(\chi)}$ defined after \eqref{Pi_on_homology}, which is Eisenstein 
	away from $\mc{N}$ in the sense of \eqref{Eis_prop} and Proposition \ref{Eis_Galois}.
\end{theorem}

\begin{remark}
	We expect that $\Pi_{\mc{N}}^{\chi}$ is Eisenstein at primes dividing the level as well. Here, this would mean that $\Pi_{\mc{N}}^{\chi}
	\circ U_{\mf{q}}^* = \Pi_{\mc{N}}^{\chi}$ for each $\mf{q}$ dividing $\mc{N}$, where $U_{\mf{q}}^*$ is a dual Hecke operator attached
	to $\mf{q}$. This should allow one to remove the condition that 
	$\mc{N}$ is divisible by at most one power of $\mf{p}$ in Theorem \ref{Eis_map_homol}. 
	That is, though we have not written down details, the argument of Proposition \ref{Eis_part_BS}
	for the Eisenstein localization of boundary homology including each $U_{\mf{p}}^*-1$ for $\mf{p}$ dividing $p$ 
	should go through with the same list of characters, as $U_{\mf{p}}^*$ should act as $0$ on 
	$H^1(\partial X_1^{\mr{BS}}(\mc{N}),\C)_{\mu}$ in the case that $\mu_1$ and $\mu_2$ are both ramified at $\mf{p}$.
\end{remark}

\vspace{0ex}

\small
\begin{multicols}{2}
\noindent
	Emmanuel Lecouturier\\
	Institute for Theoretical Sciences\\
	Westlake University\\
    No.~600 Dunyu Road, Sandun town, Xihu district \\
    310030 Hangzhou, Zhejiang, China \\
    elecoutu@westlake.edu.cn\\
	Romyar Sharifi\\
	Department of Mathematics\\ 
	University of California, Los Angeles\\ 
	520 Portola Plaza\\ 
	Los Angeles, CA 90095, USA\\
	sharifi@math.ucla.edu
\end{multicols}


\begin{multicols}{2}
\noindent
    Sheng-Chi Shih\\
    Faculty of Mathematics \\
    University of Vienna \\
    Oskar-Morgenstern-Platz 1 \\
    A-1090 Wien, Austria \\
    shengchishih@gmail.com\\ 
    Jun Wang\\
    Institute for Advanced Study in Mathematics \\
    Harbin Institute of Technology \\
    No. 92 West Da Zhi Street, Nan Gang District\\
    150001 Harbin, Heilongjiang, China\\
    junwangmath@hit.edu.cn 
\end{multicols}

\end{document}